\documentclass[reqno]{imsart}
\usepackage{color}
\usepackage{comment}
\usepackage{pdfpages}
\usepackage{graphicx}
\usepackage{dcolumn}
\usepackage{dsfont}

\RequirePackage[numbers]{natbib}
\RequirePackage[colorlinks=true, pdfstartview=FitV, linkcolor=blue,
  citecolor=blue, urlcolor=blue]{hyperref}
\usepackage{paralist}

\usepackage{tikz} 
\usepackage{dcolumn}
\usetikzlibrary{arrows,positioning}
\tikzset{
    >=stealth',
    pil/.style={
           ->,
           thick,
           shorten <=2pt,
           shorten >=2pt,}
}

\usepackage{graphicx}
\graphicspath{{./figures/}}
\usepackage{amsfonts}
\usepackage{amsmath}
\usepackage{amsthm}
\usepackage{amssymb}
\usepackage{amsbsy}
\usepackage{epsfig}
\usepackage{fullpage}
\usepackage{natbib, mathrsfs}
\usepackage{verbatim}
\usepackage[latin1]{inputenc}
\usepackage{mhequ}
\usepackage{algorithm}
\usepackage{algorithmic}

\numberwithin{equation}{section}

\def \config{\Omega_{\Lambda}}
\def \pconfig{\Omega_{\mathcal{B}(\Lambda)}}

\def \conc{\bullet}
\def \S{\mathcal{S}}

\newcommand \init[1]{#1^{(\mathrm{init})}}
\newcommand \fin[1]{#1^{(\mathrm{fin})}}

\newcommand \flowtyp[1]{F^{(\mathrm{pt})}_{#1}}

\newcommand \flowna[1]{F^{(\mathrm{na})}_{#1}}
\newcommand \flowfull[1]{F_{#1}}
\newcommand \flownew[1]{F^{(\mathrm{mid})}_{#1}}
\newcommand \flowtrunc[2]{F^{(#2)}_{#1}}
\newcommand \flowtot[1]{F^{(\mathrm{pt})}_{#1}}

\newcommand \res[2]{#2|_{#1}}

\def \rect{\mathbf{R}}
\def \edges{\mathbf{E}}

\newcommand \longnewanc[2]{G^{(\mathrm{long, new})}(#1,#2)}

\newcommand \recrit[1]{\mathcal{R}_{n}}

\def \be{\begin{equs}}
\def \ee{\end{equs}}
\def \P{\mathbb{P}}
\def \E{\mathbb{E}}

\def \T{\mathcal{T}}

\def \G{\mathcal{G}}

\def \t{\theta}

\def \L{\mathrm{L}}

\def \m{n}

 \let\b=\beta     
 \let\g=\gamma \let\h=\eta      
         \let\p=\pi  
  \let\s=\sigma \let\t=\tau   
  
   \let\G=\Gamma  \let\L=\Lambda 
\let\O=\Omega

\renewcommand{\leq}{\;\leqslant\;}                   
\renewcommand{\geq}{\;\geqslant\;}                   
\newcommand{\inftwo}[2]{\inf_{\substack{#1 \\ #2}}} 
\renewcommand{\b}{\beta}

\newcommand{\1}{\mathds{1}}

\newcommand{\var}{\operatorname{Var}}

\newcommand{\tmix}{T_{\rm mix}}

\newcommand{\gap}{{\rm gap}}

\newcommand{\N}{\mathbb N}

\newcommand{\cB}{\ensuremath{\mathcal B}}

\newcommand{\cD}{\ensuremath{\mathcal D}}

\newcommand{\cG}{\ensuremath{\mathcal G}}

\newcommand{\cL}{\ensuremath{\mathcal L}}

\newcommand{\cN}{\ensuremath{\mathcal N}}

\newcommand{\cQ}{\ensuremath{\mathcal Q}}
\newcommand{\cR}{\ensuremath{\mathcal R}}
\newcommand{\cS}{\ensuremath{\mathcal S}}

\newcommand{\bbE}{{\ensuremath{\mathbb E}} }

\newcommand{\bbP}{{\ensuremath{\mathbb P}} }

\newcommand{\bbR}{{\ensuremath{\mathbb R}} }

\newcommand{\bbZ}{{\ensuremath{\mathbb Z}} }
\newcommand{\Z}{{\ensuremath{\mathbb Z}} }

\definecolor{WowColor}{rgb}{.75,0,.75}
\definecolor{SubtleColor}{rgb}{0.9,0,0}

\newcounter{margincounter}

\newcounter{latercounter}

\newtheorem{theorem}{Theorem}[section]
\newtheorem{lemma}[theorem]{Lemma}

\newtheorem{corollary}[theorem]{Corollary}

\theoremstyle{plain}
\newtheorem{thm}{Theorem}
\newtheorem*{thm-non}{Theorem}

\theoremstyle{definition}
\newtheorem{defn}[theorem]{Definition}
\newtheorem{remark}[theorem]{Remark}

\DeclareMathOperator*{\argmin}{arg\,min}

\begin{document}

\begin{frontmatter}
\title{Cutoff for the Square Plaquette Model on a Critical Length Scale}
\runtitle{Cutoff for the Square Plaquette Model}


\begin{aug}
\author{\fnms{Paul} \snm{Chleboun}\ead[label=e1]{paul.i.chleboun@warwick.ac.uk}}
\and
\author{\fnms{Aaron} \snm{Smith}\thanksref{t2}\ead[label=e2]{asmi28@uottawa.ca}}
\thankstext{t2}{Supported by a grant from NSERC.}
\runauthor{Chleboun and Smith}
\affiliation{
	University of Warwick and
	University of Ottawa\thanksmark{m2}}
\address{
	Department of Statistics,\\
	University of Warwick,  \\
	Coventry,  CV4 7AL \\  United Kingdom \\ \printead{e1}}
\address{
	Department of Mathematics and Statistics\\
	University of Ottawa \\
	585 King Edward Avenue,  Ottawa\\
	ON K1N 7N5 \\ Canada  \\\printead{e2}}
\end{aug}

\maketitle




\begin{abstract}

Plaquette models are short range ferromagnetic spin models that play a key role in the dynamic facilitation approach to the liquid glass transition. In this paper we study the dynamics of the square plaquette model at the smallest of the three critical length scales discovered in \cite{Chleboun2017}.
Our main result is that the plaquette model with \textit{periodic} boundary conditions,  on this length scale, exhibits a sharp transition in the convergence to equilibrium, known as cutoff. This substantially refines our coarse understanding of mixing from previous work \cite{ChlebounSpm2018}.

\end{abstract}

\end{frontmatter}


\setcounter{tocdepth}{1}
\tableofcontents

\section{Introduction} \label{sec:intro}

In this paper we consider the dynamics of the square plaquette model (SPM) at low temperature.
Spin plaquette models were originally associated with glassy behaviour in \cite{Garrahan2002b,Newman1999}, where it was argued that they have more physically-realistic dynamics and thermodynamic properties than kinetically-constrained models, while remaining mathematically tractable (see \cite{Garrahan2010} for a review and \cite{Martinelli2019} for recent references on KCMs).
Understanding the liquid-glass transition and the dynamics of amorphous materials remains a significant challenge in condensed matter physics (for a review see \cite{Berthier2011a}).

Plaquette models are defined over an integer lattice $\Lambda \subseteq \Z^{d}$, and configurations of the plaquette model correspond to $\pm 1$-valued labellings of the lattice $\Lambda$. Every configuration of the SPM has an associated energy given by short range ferromagnetic interactions that are defined in terms of the  \textit{plaquettes}, a collection of subsets of $\Z^{2}$ denoted by $\mathcal{P}$.
Formally, plaquette models are families of probability distributions on $\{-1,+1\}^{\Lambda}$ described as follows. Every configuration $\sigma \in \{-1,+1\}^{\Lambda}$ has an associated energy value; this is (roughly) given by the Hamiltonian \[H_\L(\sigma) = -\frac{1}{2}\sum_{P\in \mathcal{P}}\prod_{x \in P}\sigma_x\,,\]
(the actual Hamiltonian will also depend on the boundary conditions chosen).
 For any fixed inverse-temperature $\beta > 0$, we then associate to this Hamiltonian the probability distribution $\pi_{\Lambda}^\beta$ on $\{-1,1\}^{\Lambda}$ given by $\pi_{\Lambda}^\beta(\sigma) \propto \textrm{exp}(-\beta H_\Lambda(\sigma))$.  In the case of the SPM, we always take $\Lambda \subset \Z^{2}$ and the plaquettes are exactly the collection of unit squares which intersect $\Lambda$.
Despite the relatively simple form of the Hamiltonian, the thermodynamics of these measures is non trivial for the SPM \cite{Chleboun2017}.
Although there is a unique infinite volume Gibbs measure for each $\beta$, static correlation lengths grow extremely quickly as $\beta \to \infty$, and for certain boundary conditions the ground states (configurations $\sigma$ that minimise $H(\sigma)$) are highly degenerate. 
It turns out that this plays an important role for the dynamics of the process.

In this paper we study  the associated \textit{continuous-time single-spin Glauber dynamics} (also known as the \textit{Gibbs sampler}). Roughly speaking, starting from a configuration $\s$, the local configuration $\s(x)$ at each site $x \in \Lambda$ updates at unit rate by choosing a new value according to  $\pi_{\Lambda}^{\beta}$ conditioned on $\{\sigma(y)\}_{y \in \Lambda \backslash \{x\}}$.
The dynamics of plaquette spin models has been the focus of several works in the physics literature. Initial simulations clearly indicate the occurrence of glassy dynamics (extremely slow relaxation) at low temperature \cite{Jack2005a,Newman1999}. For SPMs these results have recently been confirmed rigorously by the authors \cite{ChlebounSpm2018}.  

The products $\prod_{x \in P} \sigma_{x}$ of the spins over the individual plaquettes $P \in \mathcal{P}$ play an important role in studying these dynamics. These products  are called the \emph{plaquette variables}, and a plaquette variable equal to $-1$ is said to be a \emph{defect}.
At low temperature the dynamics of the defects are effectively constrained (see Fig. \ref{FigPairMove}). 
Simulations and heuristic analysis based on this observation suggest that for the SPM the relaxation time scales like $e^{c \beta}$ (Arrhenius scaling), and that the dynamics are closely related to those of the Fredrickson-Andersen KCM. The mixing properties of this model have been well studied (see \cite{Pillai2017,Pillai2017a,Blondel2012a} and references therein).
On the other hand, in a related model called the triangular plaquette model the relaxation time is expected to scale like $e^{c \beta^2}$ (super-Arrhenius scaling) \cite{Garrahan2002b}. 
The defect dynamics of the triangular plaquette model are closely related to a particularly KCM known as the East model which has been widely studied (see \cite{Chleboun2015,Faggionato2012,Ganguly2015} and references therein).
As has been observed for KCM the difference between Arrhenius and super-Arrhenius scaling is fundamentally due to the nature of the energy barriers that should be overcome to bring isolated defects together and annihilate them \cite{Martinelli2019,martinelli2019a}.
Although a basic heuristic in terms of energy barriers seems to be good at distinguishing between Arrhenius and super-Arrhenius behaviour, we show in this work that they do not accurately predict the constant $c$, and we are able to confirm the refined heuristic suggested in \cite{ChlebounSpm2018}.

We are primarily interested in the dynamics of the SPM in boxes $\Lambda = \{1,2,\ldots,L\}^{2}$ with periodic (toroidal) boundary conditions and side length $L$ given by what physicists call the \textit{critical length scale} -  the correlation length for the product of spin variables in the infinite volume Gibbs measure. In \cite{Chleboun2017}, this critical length scale was shown to satisfy $L \approx e^{\frac{\beta}{2}}$ as $\beta \rightarrow 0$. 
We recall that the cutoff phenomenon, coined by Aldous and Diaconis (see for example \cite{Diaconis1996} and references therein), refers to a sharp transition in the convergence to stationarity, where the distance to equilibrium drops from close to $1$ to close to $0$ over a timescale that is much shorter than the mixing time.
Our main result is that these dynamics exhibit a sharp cutoff at time $T = C \beta e^{4 \beta} (1 + o(1))$ (see Theorem \ref{ThmMainResPer}). We also substantially sharpen the bounds on the mixing of the SPM for all-plus boundary conditions obtained in \cite{ChlebounSpm2018}, showing that it is $O(\beta^{7} e^{3.75 \beta})$ (see Theorem \ref{ThmMainResPlus}). Among other consequences, this implies that the mixing has a strong dependence on the boundary conditions. 
A similar phenomenon had been previously observed for certain kinetically constrained models \cite{Chleboun2014}.

It turns out that the two different time scales, as well as the cutoff phenomenon, are caused by the structure of the ground states. In particular, with all plus boundary conditions there is a unique ground state (all sites have spin $+1$), and the low temperature dynamics are dominated by the time to reach the ground state, which we conjecture to be $e^{3.5 \beta (1 + o(1))}$.
On the other hand, with periodic boundary conditions there are $2^{2L-1}$ ground states, where $L$ is the side length of the box. These ground states correspond to flipping all the spins in any set of rows and columns with respect to the all plus state. In this case, the dynamics are dominated by an induced random walk on the ground state. This walk is very similar to the standard random walk on the hypercube, which is well-known to exhibit cutoff, and we are able to show that the induced random walk inherits cutoff at $C \beta e^{4 \beta} (1 + o(1))$.

Following our work in \cite{ChlebounSpm2018}, the main tools will be detailed canonical path bounds using multi-commodity flows \cite{Sinclair1992}, combined with the spectral profile method introduced in \cite{Goel2006}. 
Roughly speaking, the canonical paths in \cite{ChlebounSpm2018} were formed by concatenating two types of path, an initial segment that went over small energy barriers, and a final segment that was very short. 
In the present paper, we interpolate between these regimes by defining paths that go over small (but not minimal) energy barriers and have small (but not minimal) length. 
These interpolated bounds allow us to substantially sharpen our bound on the time required to enter a ground state for the first time. In the case of periodic boundary conditions, our analysis of the walk from a first ground state reached involves a sharp comparison of the trace of the process on the ground states to the simple random walk on the hypercube; from there we can use the result that mixing times are related to the hitting times of large sets \cite{Oliveira2012,Peres2015}.

\subsection{Guide to Paper}

We set basic notation for the paper in Section \ref{SecNotation}. This section also includes the main heuristics guiding our proof; we suggest that readers fully digest these heuristics before reading the precise arguments.  Section \ref{SecMainResults} includes a precise statement of our main results. Section \ref{SecConCanPath} is the bulk of the paper. It describes the canonical path method, constructs new canonical paths that will be used to analyze the SPM, and gives detailed bounds on the properties of these paths. Finally, Sections \ref{SecAllPlusRes} and \ref{SecPerBoundRes} contain the proofs of our main results: an improved bound on the mixing time for the all-plus boundary condition and a cutoff theorem for periodic boundary conditions respectively.

\section{Notation and Background} \label{SecNotation}

\subsection{Basic Conventions}

We denote the two canonical basis vectors of $\bbZ^2$ by $e_1 = (1,0)$ and $e_2= (0,1)$. We denote the projections of $x \in \bbZ^2$ on $e_1$ and $e_2$ by $x_1$ and $x_2$ respectively.
 We define the shorthand $[a\!:\!b] = \{a,a+1,\ldots,b\}$ when $b-a \in \mathbb{N}$.

Given $\L \subseteq \bbZ^2$ we will denote by $\config = \{ -1, 1\}^{\L}$ the state space of the plaquette model, endowed with the product topology.
We let $\O=\O_{\bbZ^2}$.
  Given  $A \subset \L \subseteq \bbZ^2$  with $A$ finite and a configuration $\s \in \O_\L$ we define  $\res{A}{\s}$ as the restriction of $\s$ to $A$.
We define the  \emph{plaquette} at site $x \in \bbZ^2$ to be the set of four sites $B_x = \{x,x+e_1,x+e_2,x+e_1+e_2\}$.
We also write $\overline{B}_x= B_{x-e_1-e_2} = \{x,x-e_1,x-e_2,x-e_1-e_2\}$.

For functions $f, g \, : \, \mathbb{R}^{+} \mapsto \mathbb{R}^{+}$ we write $f = O(g)$ if there exists $0 < C,X < \infty$ so that $f(x) \leq C \, g(x)$ for all $x > X$. We also write $f = o(g)$ if $\lim_{x \rightarrow \infty} \frac{f(x)}{g(x)} = 0$, and we write $f = \Omega(g)$ if $g = O(f)$. Finally, we write $f = \Theta(g)$ if both $f = O(g)$ and $g = O(f)$. To save space, we also write $f \lesssim g$ for $f = O(g)$, we write $f \gtrsim g$ for $f = \Omega(g)$, and we write $f \approx g$ for $f = \Theta(g)$. Similarly, to save space, all inequalities should be understood to hold only for all $\beta > \beta_{0}$ sufficiently large. For example, we may write $e^{\beta} \geq \beta + 3$ without additional comment. Since all of our results are asymptotic as $\beta$ goes to infinity, this convention will not cause any difficulties. For any function $f:A\to B$, we denote  the image of $f$ by $f(A)$.

For a transition rate matrix $K$ on state space $\Omega$, denote by $\edges = \edges_{K} = \{(x,y) \in \Omega^{2} \, : \, K(x,y) > 0\}$ the collection of transitions with non-zero rates under $K$. This collection of edges defines a graph. when the rate matrix $K$ is clear from the context, we sometimes write $\edges$ for $\edges_{K}$ and discuss subsets of $\Omega$ using graph-theoretic definitions such as a \textit{connected component} in $\Omega$.

Finally, we define two orders on $\mathbb{Z}^{2}$. The first, which we will refer to as the \textit{lexicographic order}\footnote{There are several different ``lexicographic" orders in the literature. The order in this paper corresponds to the order in which words are read in English if the Cartesian plane is drawn in the usual way.}, is as follows: for $x \neq y \in \mathbb{Z}^{2}$, we say $x <_{\mathrm{lex}} y$ if and only if one of the two following  conditions hold:
\begin{enumerate}
	\label{eq:lexi}
	\item $x_{2} > y_{2}$, or
	\item $x_{2} = y_{2}$ and $x_{1} < y_{1}$.
\end{enumerate}
Similarly, we define the \textit{anti-lexicographic order} by saying that $x <_{\mathrm{xel}} y $ if and only if one of the two following  conditions hold:
\begin{enumerate}
	\item $x_{2} > y_{2}$, or
	\item $x_{2} = y_{2}$ and $x_{1} > y_{1}$.
\end{enumerate}
By a small abuse of notation, we say that a set $S_{1}$ is less than a set $S_{2}$ in lexicographic order if \textit{every} element of $S_{1}$ is less than \textit{every} element of $S_{2}$.

\subsection{Equilibrium Gibbs measures}
We will define the finite volume Gibbs measures on $\L  \subset \bbZ^2$ with  fixed and periodic boundary conditions.
Let $\cB(\L) = \{x\in \bbZ^2 \,:\, B_x  \cap \L \neq \emptyset\}$ be the set of plaquettes which intersect $\Lambda$, indexed by their bottom left vertex, and
let $\cB_{-}(\L) = \{x \in \bbZ^2 \, : \, \overline{B}_{x} \subset \L\}$ (if $\L$ is a rectangle rectangle then $\cB_{-}(\L)$ is just $\L$ without the left most column and bottom row).
For a boundary condition $\t \in \O$, wedefine $\O_\L^{\t} = \{\s\in \O \,:\, \res{\L^c}{\s} \equiv \res{\L^c}{\t}\}$. Finally, we denote the external boundary of $\L$ by  $\partial(\L) = \cup_{x\in\cB(\L)}B_x\setminus  \L$.

For fixed boundary conditions $\t$, the plaquette variables associated with a spin configuration are defined by the map $p^{\t}:\O_\L^{\t}\to \O_{\cB(\L)}$ which is given by the formula
\begin{align}
  \label{EqDefectMap}
  p^{\t}_x(\s) = \prod_{y \in B_x}\s_y= \s_{(x_1,x_2)}\s_{(x_1+1,x_2)} \s_{(x_1,x_2+1)}\s_{(x_1+1,x_2+1)}\,, \quad \textrm{for } x \in \cB(\L)\,.
\end{align}
Similarly,  for periodic boundary conditions on a box $\L=[0\!:\!L_1-1]\times [0\!:\!L_2-1]$, define $p^{\rm per}:\O_\L \to \O_{\L}$ by
\begin{align}
  p^{\rm per}_x(\s) = \prod_{y \in B_x}\s_y= \s_{(x_1,x_2)}\s_{(x_1+1,x_2)} \s_{(x_1,x_2+1)}\s_{(x_1+1,x_2+1)}\,, \quad \textrm{for } x \in \L\,,
\end{align}
where the sums $x_1{+}1$ and $x_2{+}1$ above are taken modulo $L_1$ and $L_2$ respectively.
We say there is a \emph{defect} in $\s$ at $x \in \cB({\L})$ if $p^{\t}_x(\s) = -1$ (similarly for periodic boundary conditions). By a small abuse of notation, we consider ``per" to be a boundary condition.

For $\s \in \O_\L$ or $\O_\L^\t$ let $|\s| = |\{x \in \L \,:\, \s_x = -1\}|$ denote the number of minus spins and $|p^{\t}(\s)| =  |\{x \in \cB(\L)\,:\, p^{\t}_x(\s) = -1\}|$ the number of defects (similarly for $|p^{\rm per}(\s)|$).
We define a partial order on plaquette variables with respect to defects by
\begin{align}
  p^\t(\s) \leq p^\t(\h) \iff \{x \in \cB(\L)\,:\, p^{\t}_x(\s) = -1\} \subset \{x \in \cB(\L)\,:\, p^{\t}_x(\h) = -1\}\,,
\end{align}
similarly for periodic boundary conditions.

We define the Hamiltonian $H_\L^{\t}: \O_{\L}^{\t}\to \bbR$ with boundary condition $\t$ by
\begin{align}
\label{eq:hamtau}
  H^{\t}_\L(\s) = -\frac{1}{2}\sum_{x\in \cB(\L)}p^{\t}_x(\s)\,,
\end{align}
and similarly for periodic boundary condition. The finite volume Gibbs measure on $\L$ with boundary condition $\t$ is then denoted by $\p_\L^{\t}$ and given by
\begin{align}
 \label{eq:pitau}
  \pi_{\L}^{\t}(\s) = \frac{e^{-\b H^{\t}_\L(\s)}}{Z^{\t}_\L(\b)}\,,
\end{align}
where $Z^{\t}_\L(\b) =\sum_{\s\in\O_\L^{\t}} e^{-\b H^{\t}_\L(\s)}$ is the partition function. The analogous formula gives the finite volume Gibbs measure $\pi_{\L}^{\rm per}$ for periodic boundary conditions. For brevity, if $\t \equiv \pm 1$ we will replace $\t$ with $\pm$, for example with plus boundary conditions we will write $H_\L^+$, $\pi^+_\L$.
Also, where there is no confusion, we denote by $+\in \O$ the configuration of all $+1$ spins.
When the boundary conditions and lattice are clear from the context, we may drop the boundary condition superscript and the lattice subscript.

It turns out that the size of $(p^{\t})^{-1}(p)$ is independent of $p \in p^\t(\O^\t_\L)$ for for any boundary condition $\t \in \O$ or $\t = \textrm{per}$ (see Section 4 of \cite{ChlebounSpm2018}), in particular 
\begin{align} \label{eq:defectprob}
\pi^{\t}_\L(\{p^{\t}(\s) = p\})\propto e^{-\beta |p|}\1_{p^\tau(\O_\L^\t)}(p),\     \textrm{ for } p \in \O_{\cB(\L)}\,.
\end{align}
Since $|p^{+}(+)| = |p^{\rm{per}}(+)| = 0$, we also have 
\begin{align}
  \label{eq:simple}
  \pi_\L^{\t}(\s) =  \pi_\L^{\t}(+) e^{-\b |p^{\t}(\s)|}\,,\quad \textrm{ for } \s \in \O_{\L}^{\t},\quad \textrm{ and } \tau \in \{+,\rm{per}\}\,.
\end{align}
For boundary conditions $\t \in \{+,\textrm{per}\}$ we note that the plaquette configurations always satisfy a parity constraint: the number of defects in any row or column of $p^\t(\s)$ is even for each $\s \in \O_\L^\t$ (see Section 4 of \cite{ChlebounSpm2018}).

For a given boundary condition $\tau$, define the collection of \textit{ground states} by:
\begin{align}
\cG = \cG^{\t} = \{ \sigma \in \O_\L^{\t} \, : \, H_\L^{\t}(\sigma) = \min_{\eta \in \O_\L^\t} H_\L^{\t}(\eta) \}.
\end{align}
Finally we will frequently use the fact that ground states dominate the stationary distribution.
\begin{lemma}[Domination of Ground States]
\label{lem:dom}
  Let $\t \in \{+,\rm{per}\}$ and $L=L_c$, then
  \begin{align}
    \pi(\cG) \geq 1 - O(e^{-2\b})\,.
  \end{align}
\end{lemma}

\begin{proof}
This is Lemma 4.4 of \cite{ChlebounSpm2018}.
\end{proof}

\subsection{Finite volume Glauber dynamics}

For a set $S \subset \L$ and $\s \in \config$, denote by $\s^{S}$ the configuration obtained by flipping all the spins of $\s$ that lie in $S$. With slight abuse of notation, we define $\s^{x} = \s^{\{x\}}$ for $x \in \L$.

Given a finite region $\L$ and boundary condition $\tau$, we consider the continuous time Markov process determined by the generator
\begin{align}
\label{eq:gen}
  \cL_{\L}^{\t} f(\sigma)=\sum_{x\in\L}c^{\tau}_{\L}(x,\sigma)(f(\s^x)-f(\s)) = \sum_{x\in\L}c^{\tau}_{\L}(x,\sigma)\nabla_xf(\s)\,,
\end{align}
where we define $\nabla_xf(\s) = (f(\s^x)-f(\s))$, and where the Metropolis spin-flip rates $c^{\tau}_\L(x,\sigma)$ are given by the formula
\begin{equation}
\label{eq:rates}
c^{\tau}_\L(x,\sigma)=
\begin{cases}
e^{-\beta(H_{\L}^{\t}(\sigma^x)-H_{\L}^{\t}(\sigma))} & \text{if } H_{\L}^{\t}(\sigma^x)>H^{\t}_{\L}(\sigma)\,,\\
1 & \text{otherwise.}
\end{cases}
\end{equation}
With a slight abuse of notation, we denote  the elements of the associated transition rate matrix by by $\cL_{\L}^{\t}(\s,\h)$, for $\s,\h \in \O_\L^\t$.
The process is  reversible with respect to the finite volume equilibrium measure $\pi_{\L}^{\t}$.
\begin{remark}
All our results hold equally well for the standard heat-bath dynamics, since $\left(1+e^{\beta(H_{\L}^{\t}(\sigma^x)-H_{\L}^{\t}(\sigma))} \right)^{-1} \approx \min\{e^{-\beta(H_{\L}^{\t}(\sigma^x)-H_{\L}^{\t}(\sigma))},1\}$ for large $\beta$.
\end{remark}

Since $H^\t_\L(\s)$ only depends on the plaquette variables, the spin dynamics also induce a dynamics on these ``defect" variables which is Markov. The generator of the defect dynamics is given by
\begin{align}
  \label{eq:genplaq}
  \mathcal{Q}_\L^\t f(p) = \sum_{x\in \L} k^\t_\L(x,p)\left(f(p^{\overline{B}_{x}})- f(p)\right),
\end{align}
where we recall
\begin{align*}
  p^{\overline{B}_{x}}_z =
  \begin{cases}
    -p_z & \textrm{if } z \in \overline{B}_{x} = \{x-(1,1),x-(0,1),x-(1,0),x \}\,,\\
    p_z & \textrm{otherwise.}
  \end{cases}
\end{align*}
From \eqref{eq:rates}, the transition rates for this process are given by
\begin{align}
  \label{eq:ratesp}
  k^\t_\L(x,p) = \min\left\{ {\rm exp}\left[-\beta\left(|p^{\overline{B}_{x}}|-|p| \right)\right], 1\right\}\,.
\end{align}

\subsection{Spectral gap}
The Dirichlet form associated with $\cL_{\L}^{\t}$ is denoted by $\cD_\L^\t(f) = -\pi_{\L}^\t(f\cL_{\L}^{\t} f)$, and it satisfies the formula
\begin{align}
  \label{def:dir}
  \cD_\L^\t(f) = \frac{1}{2}\sum_{\h\in\O}\sum_{x\in \L}\pi_{\L}^\t(\h)c_\L^\tau(x,\h)\left(\nabla_x f(\h)\right)^2\,.
\end{align}
Define $\var_\L^\t(f)$ to be the variance of $f$ with respect to $\pi_\L^\t$.

\begin{defn}[Relaxation time]
The smallest positive eigenvalue of $-\cL^\t_\L$
is called the spectral gap and it is denoted by
$\gap(\cL_{\L}^\t)$. It satisfies the Rayleigh-Ritz variational principle
\begin{align}
\label{eq:gap}
\gap ( \cL^\t_\L):= \inftwo{f \,:\,\O_\L \mapsto \bbR}{f \text{ non constant} }  \frac{ \cD_\L^\t(f) }{\var_\L^\t(f) }\,.
\end{align}
The relaxation time $T^\t_{\rm rel} (\L)$ is defined as the inverse of the spectral gap:
\begin{equation}\label{rilasso}
T_{\rm rel}^\t  (\L)= \frac{1}{ \gap (\cL^\t_\L)}\,.
\end{equation}
If $\L = [1:L]^2$ we simply write $T_{\rm rel}^\t  (L)$.
\end{defn}

\begin{defn}[Mixing time]
For fixed $0 < \epsilon < 1$, define
\be
T_{\rm mix}^{\t}(\L, \epsilon) = \inf \{s > 0 \, : \, \max_{\sigma \in \config} \| e^{s \,\cL^{\t}_{\L}}(\sigma,\cdot) - \pi_{\L}^\t(\cdot) \|_{\rm TV} < \epsilon \};
\ee
call $T_{\rm mix}^{\t}(\L) \equiv T_{\rm mix}^{\t}(\L, 0.25)$ the \textit{mixing time} of $\cL^\t_\L$. If $\L = [1,L]^{2}$ we simply write $T_{\rm mix}^\t  (L)$.

\end{defn}

\begin{defn} [Cutoff Phenomenon] \label{DefCutoff}

Consider  sequences $\beta = \{ \beta(n) \}_{n \in \mathbb{N}}$ of temperatures, $\L = \{ \L (n) \}_{n \in \mathbb{N}}$ of lattices, and $\t = \{ \t (n) \}_{ n \in \mathbb{N}}$ of boundary conditions. We say this sequence exhibits \textit{cutoff} if, for all $0 < \epsilon < 0.5$,

\be
\lim_{n \rightarrow \infty} \frac{T_{\rm mix}^{\t (n)}(\L (n), \epsilon)}{T_{\rm mix}^{\t (n)}(\L (n), 1- \epsilon)} = 1.
\ee
\end{defn}

We define the critical scale as the correlation length for the product of spin variables in the infinite volume Gibbs measure; see \cite{Chleboun2017} for further details and other important length scales.

\begin{defn}[The critical scale]
We define the critical length scale by $L_c = \lfloor e^{\frac{\b}{2}} \rfloor$.
\end{defn}

\subsection{Heuristics for Mixing and Cutoff } \label{SecHeuristic}

Recall that the \textit{ground states} are the configurations $\sigma$ that minimize $H_\L^{\t}$. For the plus and periodic boundary conditions, these are exactly the configurations with no defects. For plus boundary conditions there is a unique configuration with no defects, and for periodic boundary conditions there are $2^{2L-1}$ such configurations. 
In the case of periodic boundary conditions, we can define a bijection $w \, : \, \cG \mapsto \{-1,+1\}^{2L-1}$ from ground states to the $(2L-1)$-dimensional hypercube according to the following formula for its inverse:
\be
w^{-1}(v) [i,j] &= v[i] v[L+j], \qquad i \in [1:L], \, j \in [1:(L-1)], \\
w^{-1}(v)[i,L] &= v[i], \qquad \qquad \quad \,  i \in [1:L].
\ee

Furthermore, under this map, the minimal-length paths between ground states (which are also minimal-energy paths) all travel between configurations $\sigma, \eta$ that satisfy $|w(\sigma) - w(\eta)| = 1$ or $w(\s) = -w(\h)$. These minimal-length paths all look quite similar (see Figure \ref{FigDominantTransitionsBetweenGroundStates} for a typical example and Definition \ref{DefMinEnMinLPaths} for a complete enumeration of all minimal-length paths). If we denote by $\{ \hat{X}_{t}\}_{t \geq 0}$ the \textit{trace} of the SPM process on the collection $\cG$ of ground states, then the transition rate kernel $Q_{\mathcal{G}}^{\rm per} \circ w^{-1}$ of the Markov process $\{ w(\hat{X}_{t}) \}_{ t \geq 0}$ turns out to be quite close to the transition rate matrix $Q_{\rm H}$ of a simple group walk on the hypercube $\{-1,+1\}^{2L-1}$ with generating set $\{(-1,1,1,\ldots,1),(1,-1,1,\ldots,1),\ldots,(1,1,1,\ldots,-1),(-1,-1,-1,\ldots,-1)\}$. Furthermore, we can get an \textit{extremely} good approximation by considering low-degree polynomials in $Q_{\rm H}$: as shown in Corollary \ref{LemmaHyperComp}, there exists a sequence $c= c(\beta) = O(L^{-1})$ such that $P_{\rm H} = Q_{\rm H} + c Q_{\rm H}^2$ satisfies 
\be
\sup_{u \in \{-1,1\}^{2L-1}} \| Q_{\mathcal{G}}^{\rm per}(w^{-1}(u),\cdot)  -P_{\rm H}(u,\cdot) \|_{\mathrm{TV}} = O(L^{-9})\,.
\ee
Given this approximation, it is easy to check that the trace chain inherits cutoff from $Q_{\rm H}$.

To see that the \textit{full} Markov chain with periodic boundary conditions inherits cutoff from the \textit{trace} chain, it is enough to check:

\begin{enumerate}
\item The distribution of the transition time between ground states does not have heavy tails, and
\item Regardless of the initial configuration, the chain enters a ground state within time $o(e^{4 \beta})$ with high probability.
\end{enumerate}

The first is straightforward. The second is suggested by the following \textit{random walk heuristic}, which also appears in Section 2.5 of our previous paper \cite{ChlebounSpm2018}. Call a configuration that is not a ground state \textit{metastable} if no plaquette has more than one defect, and \textit{unstable} otherwise. As the unstable states are short-lived, we concentrate on the transitions between ``nearby'' metastable states.

Due to parity constraints on the plaquette configurations, the lowest-energy metastable configurations have exactly four defects, placed at the vertices of a rectangle. Starting the SPM process in such a ``rectangular'' configuration, with height and width of the associated rectangle at least 4, it is overwhelmingly likely that the next metastable configuration will be another ``rectangular'' configuration, with either height or width changed by exactly 1. The typical intermediate dynamics between such metastable states are shown in Figure \ref{FigPairMove}.

\begin{figure}[htb]
\includegraphics[width=0.95\linewidth]{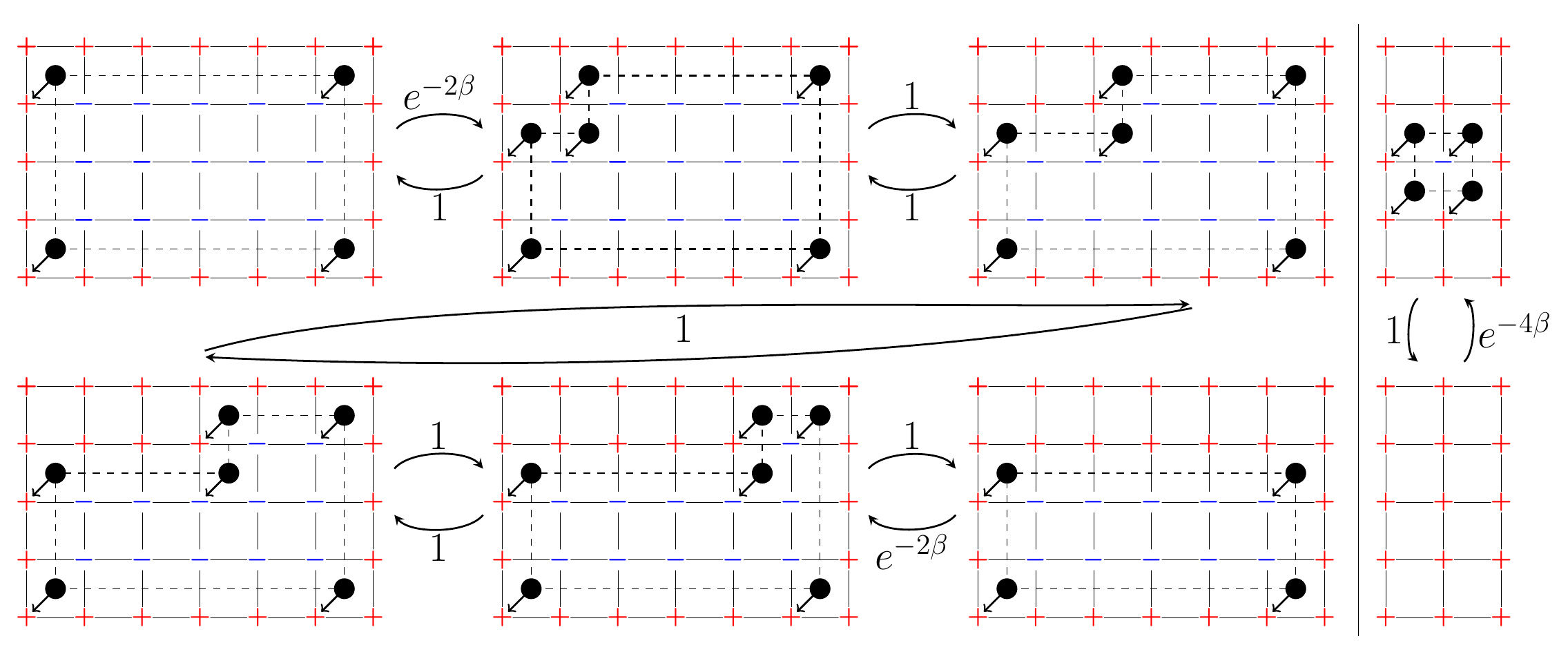}
\caption{\label{FigPairMove} The spin configuration is represented by $+$ (red) and $-$ (blue) on the lattice. The black circles represent defects ($-1$ plaquette variables), which are associated with the vertex at the bottom left of the corresponding plaquette. Dashed lines separate regions of $+$ and $-$ spins. The arrows indicate the associated transition rates. Left: An isolated defect creates two defects at rate $e^{-2\beta}$; subsequently a pair of defects is emitted and moves along an edge of the rectangle according to a simple random walk, and is then annihilated upon colliding with another defect (possibly the defect that emitted the pair). Right: the only type of transition not included on the left is to  add or remove four neighbouring defects. }
\end{figure}

Note that the ``rectangular'' configurations are entirely determined by the upper-left and lower-right vertices. Following the heuristic of Figure \ref{FigPairMove}, these corner defects perform nearly-independent simple random walks as shown in Figure \ref{FigCornerWalk}.

\begin{figure}[h]
\includegraphics[width=0.85\linewidth]{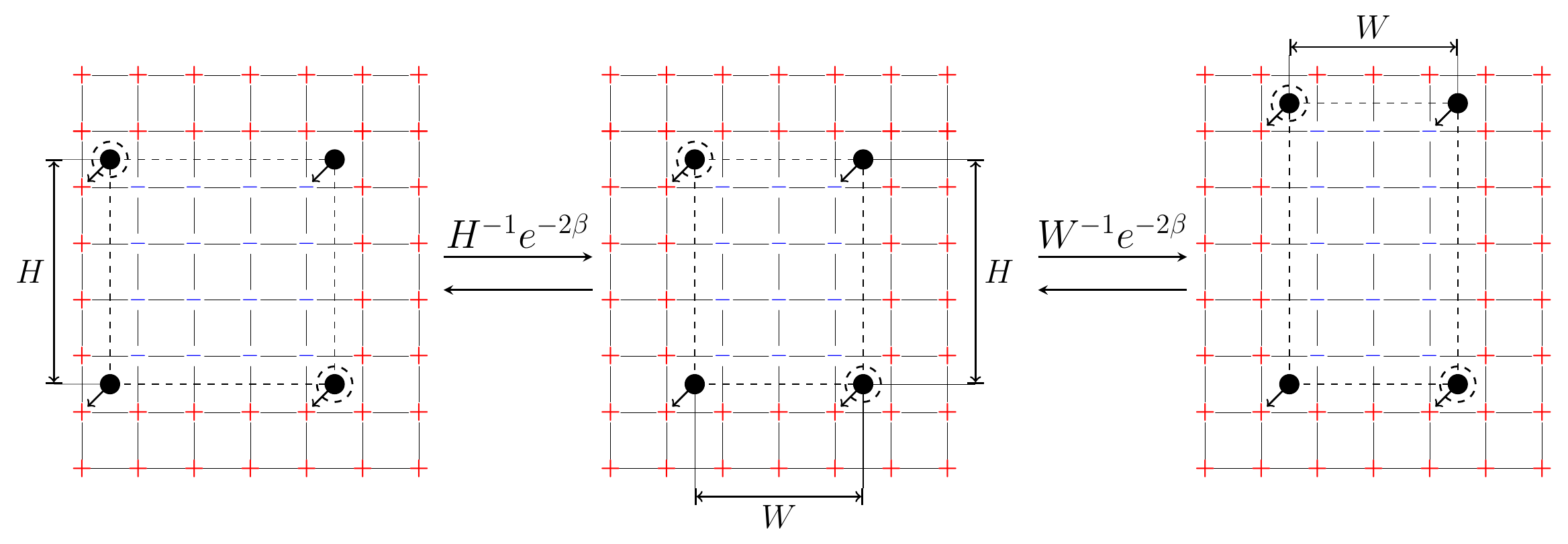}
\caption{\label{FigCornerWalk} The defects at the corners of a box approximately perform random walks. They move left-right at rate $H^{-1} e^{-2 \beta}$ and up-down at rate $W^{-1} e^{-2 \beta}$, where $H$ is the height of the box and $W$ is its width.}
\end{figure}

A routine SRW calculation says that this rectangle will collapse to one of the ground states in a time which is $\Theta \left( e^{3.5 \beta} \right)$.The canonical path method (see Section \ref{SecCanonicalPathAbs}) can be used to extend this heuristic to configurations with many defects. If we restrict our attention to configurations whose defects lie at the vertices of non-overlapping rectangles, the following canonical path construction would again give an $O(e^{3.5 \beta})$ bound on the relaxation time of the full process:

\begin{enumerate}
\item Pick a rectangle uniformly at random.
\item Follow a minimal-length path that collapses this rectangle without ever adding more than 2 defects to the initial configuration, as given by the heuristic in Figures \ref{FigPairMove} and \ref{FigCornerWalk}.
\item If any defects remain, go back to step \textbf{(1)}.
\end{enumerate}

In other words, if the rectangles don't overlap, we can essentially treat them as evolving independently. 
Unfortunately, this simple path gives very poor estimates when the rectangles do overlap. Much of this paper, like much of the precursor paper \cite{ChlebounSpm2018}, is concerned with using a canonical path and spectral profile argument to  salvage this heuristic. In \cite{ChlebounSpm2018}, we did this with a coarse ``two-stage" path construction:
\begin{enumerate}
\item In the initial stage, we successively find a ``good" collection of  rectangles with at least $3$ defects at their vertices, then choose one rectangle from this collection at random and collapse it. This construction gives a long path with low energy and congestion.
\item After collapsing $\Theta(L)$ rectangles in this way, we simply flip all the spins to $+$ in lexicographic order. This construction gives a very short path with extremely high energy and congestion.
\end{enumerate}

In that paper, most of the work was in finding the right definition for our ``good" rectangles. In the present paper, we add an intermediate collection of paths in which we collapse random ``L-shaped" configurations of defects with nearly-minimal area. These paths have lengths, energies and congestions that are intermediate between the other two and give a way to interpolate our bounds. Further details on these heuristics are in Section \ref{SubsecPathNew}.

\section{Main Results} \label{SecMainResults}

The following bound on the mixing time of the square plaquette model with all-plus boundary conditions refines Theorem 1 of \cite{ChlebounSpm2018}:

\begin{thm} [Mixing Times for Plus Boundary Conditions] \label{ThmMainResPlus}

The square plaquette process with all plus boundary conditions, on the critical scale,  satisfies

\be \label{IneqPlusMix}
e^{3.5 \beta} \lesssim T_{\rm mix}^+( L_c ) \lesssim \beta^{8} e^{3.75 \beta}.
\ee
\end{thm}

We use this to prove our main result:

\begin{thm} [Cutoff for Periodic Boundary Conditions] \label{ThmMainResPer}

The square plaquette process with periodic boundary conditions, on the critical scale, exhibits the cutoff phenomenon (see Definition \ref{DefCutoff}). Furthermore, there exists $0 < C < \infty$ such that
\be
T_{\rm mix}^{\rm per}( L_c ) = C \beta e^{4 \beta} (1 + o(1)).
\ee

\end{thm}

\begin{remark}
The spectral gaps of both chains were computed in \cite{ChlebounSpm2018}. The constant $C$ could be computed from well-known quantities related to simple random walks by carefully calculating the expectations in the random walk coupling of Section \ref{SecPerBoundRes}.
\end{remark}

\section{Construction and Analysis of Canonical Paths} \label{SecConCanPath}

Our arguments for the upper bounds in Theorems \ref{ThmMainResPlus} and \ref{ThmMainResPer} will be based on the method of canonical paths \cite{Sinclair1992}. The idea in this method is to construct a family of (possibly random) paths between any pairs of configurations $\sigma, \eta$, such that the paths do not ``congest'' too heavily on any edge.  In this section, we construct the canonical paths that will be used for those proofs, and also give some initial analysis of their properties. As a guide to the remainder of this section:

\begin{itemize}
\item In Section \ref{SecCanonicalPathAbs}, we recall generic bounds on relaxation and mixing times of a Markov chain in terms of canonical paths.
\item In Section \ref{SubsecPrelPathNot}, we give preliminary notation related to canonical paths.
\item In Section \ref{SubsecPathOld}, we recall the constructions of the main canonical paths used in \cite{ChlebounSpm2018}.
\item In Section \ref{SubsecPathNew}, we construct new ``interpolating" paths and combine them with the paths constructed in Section \ref{SubsecPathOld}.
\item In Section \ref{SecFinalBoundCanonical}, we combine the results of these sub-sections to obtain final bounds on the properties of the canonical paths we have defined.
\end{itemize}

\subsection{Canonical Path Bounds} \label{SecCanonicalPathAbs}

We specialize some earlier canonical path bounds to the context of this paper. Throughout this section and remainder of the paper, the set of edges $\edges$ is taken with respect to the generator $\cL_{\L}^{\t}$ from Equation \eqref{eq:gen}.

\begin{defn}[Path]

A sequence $\gamma = (\h^{(0)},\h^{(1)},\ldots,\h^{(m)}) \in \config$ is called a \textit{path from $\h^{(0)}$ to $\h^{(m)}$}  if  $(\h^{(i-1)},\h^{(i)}) \in \edges$ for all $1 \leq i \leq m$. We say that this path has length $|\gamma| \equiv m$. For $1 \leq i \leq m$, we call the pair $(\h^{(i-1)},\h^{(i)})$ the \textit{$i$'th edge of $\gamma$}. For $\h,\s \in \Omega$, we denote by $\Gamma_{\h,\s}$ the collection of all paths from $\h$ to $\s$. Similarly we let $\Gamma_{\h} = \cup_{\s} \Gamma_{\h,\s}$ be the collection of all paths starting at $\h$ and $\Gamma = \cup_{\h,\s \in \config} \Gamma_{\h,\s}$ be the collection of all paths.

For a path $\gamma = (\h^{(0)},\ldots,\h^{(m)})$, we denote by $\init{\gamma} = \h^{(0)}$ and $\fin{\gamma} = \h^{(m)}$ the \textit{initial} and \textit{final} elements of $\gamma$. If $\gamma_{1}$, $\gamma_{2}$ are two paths with $\fin{\gamma_{1}} = \init{\gamma_{2}}$, we denote by $\gamma_{1} \conc \gamma_{2}$ the concatenation of $\gamma_{1}$ and $\gamma_{2}$ with the repeated element $\gamma_{1}^{(fin)}, \gamma_{2}^{(init)}$ removed. With some abuse of notation, we define $\emptyset \conc \gamma = \gamma \conc \emptyset = \gamma$ for any path $\gamma$.
\end{defn}

Next, for fixed $r>0$, let

\be \label{DefHighDensityRound}
k(r) = \min \{ k \, : \, \pi(+)  \, e^{-\beta k} \leq r \}.
\ee
For fixed $k \in \mathbb{N}$, let
\be \label{DefHighDensitySet}
S_{k} = \{ \sigma \in \config \, : \, |p^{+}(\sigma)| \geq k \}.
\ee

Define the \emph{cost} of a flow:

\begin{defn} [Cost of Flow] \label{DefCostFlow}

Fix $k \in \mathbb{N}$. For each $\h \in S_{k}$, let $F^{(k)}_{\h}$ be a probability measure on paths $\gamma$ in $\config$ that have starting point $\init{\gamma} = \h$ and final point $\fin{\gamma} \in S_{k}^{c}$. We define the \textit{cost} of the flow, $\{F^{(k)}_\h\}_{\h \in \S_k}$ on $S_{k}$, to be:
\be \label{IneqDefCanPathObject}
\mathcal{A}(k) \equiv 2\, \max_{e \in \edges} \sum_{\h \in \config} \sum_{\gamma \ni e} F^{(k)}_{\h}(\gamma) \, | \gamma | \, \frac{\mu(\h)}{\mu(e_{-}) K(e_{-},e_{+})}.
\ee
Note that we have dropped the explicit dependence on $\{ F^{(k)}_{\h} \}_{\h \in S_{k}}$ from this notation for convenience.
\end{defn}

The following is an immediate consequence of Theorem 1.1 of \cite{Goel2006} combined with the multicommodity flow bound in Lemma 5.2 of \cite{ChlebounSpm2018} and the reduction in Inequality (5.7) of \cite{ChlebounSpm2018}:

\begin{theorem} \label{ThmMainSpectralProfileBound}
Fix for all $S_{k}$, $k \in \mathbb{N}$ a collection of paths $\{ F^{(k)}_{\h}\}_{\h \in S_{k}}$  as in Definition \ref{DefCostFlow}. Then

\be \label{IneqMainSpectralProfileBound}
\tmix \leq \int_{4 \, \min_{\s \in \O}{\pi(\s)}}^{16} \frac{2 \mathcal{A}(k(x))}{x} dx.
\ee

\end{theorem}

\subsection{Preliminary Notation} \label{SubsecPrelPathNot}

We will use the following notation for defects frequently:

\begin{defn} [Neighbouring Defects] \label{DefNeighbourDefect}

Let $\Lambda = [\ell_{1}:\ell_{2}] \times [\ell_{3}:\ell_{4}]$. Fix $p \in \pconfig$ and define
\be
D(p) = \{x \in \pconfig \,: \, p_{x} = -1 \}.
\ee

For $p \in \pconfig$ and $x = (x_{1},x_{2}) \in D(p)$, define the \textit{left row neighbour}, \textit{right row neighbour}, \textit{down column neighbour} and \textit{up column neighbour} of the defect at $x$ by
\be  \label{EqNearestNeighbours}
\ell(x) &= (\max\{x_1' < x_1\,:\, (x_1',x_2) \in D(p)\},x_{2})\,, \\
r(x) &= (\min\{x_1' > x_1\,:\, (x_{1}',x_{2}) \in D(p)\}, x_{2})\,, \\
d(x) &= (x_{1},\max\{x_2'<x_2\,:\, (x_{1},x_{2}') \in D(p)\})\,, \\
u(x) &= (x_{1},\min\{x_{2}'> x_2\,:\,(x_1,x_2') \in D(p)\})\,,
\ee
when they exist; otherwise by $\emptyset$ as appropriate.
Finally, by a small abuse of notation we define
$
D(\sigma) = D(p^{+}(\sigma))
$
for $\sigma \in \O_\L^+$.
\end{defn}

We now define a simple ``base path" that flips all the spins in a rectangle $R \subset \Z^{2}$ in a sensible order. Typically in application, the rectangle $R$ in the following definition will have at least three defects at its vertices in the initial configuration $\sigma$, and this ``base path'' removes two or four of them without ever adding more than two in intermediate stages:

\begin{defn} [Rectangle-Removal Path] \label{DefRectRemovePath}
Let $\Lambda = [\ell_{1}\!:\!\ell_{2}]\times[\ell_{3}\!:\!\ell_{4}]$. Given $\sigma \in \config$ and a rectangle $R = [x_{1}\!:\!x_{2}] \times [y_{1}\!:\!y_{2}] \subset \cB(\Lambda)$ we define a path $\gamma_{\sigma,R}$, starting at $\sigma$, that flips all the spins at all the sites in $\cB_{-}(R)$. We construct the path according to the following cases:
\begin{enumerate}
\item If $(x_{1},y_{1}), (x_{1},y_{2}), (x_{2},y_{2}) \in D(\sigma)$, \textit{or} fewer than 3 of the four corners of the rectangle are in $D(\sigma)$, we define $\gamma_{\sigma,R}$ to be the path that flips all spins in $\cB_{-}(R)$ in lexicographic order.
\item $(x_{2},y_{2}), (x_{1},y_{2}), (x_{2},y_{1}) \in D(\sigma)$  but $(x_{1},y_{1}) \notin D(\sigma)$, we define $\gamma_{\sigma,R}$ to be the path that that flips all spins in $\cB_{-}(R)$ in antilexicographic order.
\item Otherwise, denote by $M \, : \, \config \mapsto \config$ the ``mirror reflection" map
\be \label{EqMirrorMap}
\rm{M}(\sigma)_{i,j} = \sigma_{i,\ell_{4} + \ell_{3} - j}\,
\ee
which flips the lattice in the $x$-axis.
Note that $M(R)$ is now in one of the two previous cases, and we define $\gamma_{\sigma,R} = M^{-1}(\gamma_{M(\sigma), M(R)})$.
\end{enumerate}
\end{defn}

We observe that after flipping the spins in a region $\cB_{-}(R)$ the defect configuration is changed only at the plaquettes that contain an odd number of spins in $\cB_{-}(R)$, i.e. at the four vertices of the rectangle $R$.
Furthermore, although a \textit{spin configuration} $\eta \in \gamma_{\sigma,R}$ can be very different from $\sigma$, their associated \textit{defect configurations} $D(\eta), D(\sigma)$ can only differ in a few locations. More precisely, the sites that $\eta, \sigma$ differ can be written as a union of (at most) 2 rectangles, say $A$ and $B$; then $D(\sigma), D(\eta)$ can differ (at most) at the vertices of $\cB(A)$ and $\cB(B)$.

Denote by $\rect_{\L}$ the collection of rectangles in the lattice $\L$. For $R = [x_{1}: x_{2}] \times [y_{1} : y_{2}] \in \rect_{\L}$, we informally call $y_{2}-y_1$ the ``height'' and $x_{2} - x_{1}$ the ``width'' of $R$. 
Finally, For any set $S$, denote by $\mathcal{P}(S)$ the power set of $S$. For a set $A \subset \L$ of points on the lattice we define $R(A)$ as the smallest rectangle containing $A$.

\subsection{Path Construction: Previous Constructions} \label{SubsecPathOld}

The main challenge in \cite{ChlebounSpm2018} was the construction of a reasonable ``initial" part of a canonical path, and in particular the choice of a good measure on rectangles. We will use this initial part again in the present paper. Define $F \, : \, \config \times [1:L] \mapsto \mathcal{P}(\rect_{\L})$, and $n \, : \, \config \mapsto [1:L]$, to be the two functions given in Definitions 5.15 and 5.16 of \cite{ChlebounSpm2018}, respectively.

Roughly speaking, $n(\sigma) \approx \frac{|D(\sigma)|}{100 L}$. $F(\sigma,i)$ is more complicated: for the vast majority of configurations $\sigma \sim \pi_{\L}$, the set $\cup_{i} F(\sigma,i)$ consists of essentially all rectangles with at least 3 defects in their corners. For a very small but significant collection of ``bad" configurations, $\cup_{i} F(\sigma,i)$ instead picks out a smaller subset of distinguished rectangles. See the discussion around in Section 4.3 of \cite{ChlebounSpm2018} for details and heuristics.

The canonical path section from $\sigma$ we use is to simply choose $i \in [1 : n(\sigma)]$ uniformly at random, and then $R \in F(\sigma,i)$ independently uniformly at random,  then ``collapse" this rectangle according to the path constructed in Definition \ref{DefRectRemovePath}:

\begin{defn}[Partial Random Path] \label{DefPartialPath}

For $\sigma \in \Omega_{[1:L]^{2}}$, define the ``partial path" probability measure $\flowtot{\sigma}$ by
\be
\flowtot{\sigma}(\gamma) = \frac{1}{\m(\sigma) \, |F(\sigma,i)|}
\ee
if $\gamma = \gamma_{\sigma,R}$ for some $R \in F(\sigma,i)$, and 0 otherwise.\footnote{As shown in \cite{ChlebounSpm2018}, $F(\sigma,i) \cap F(\sigma,j) = \emptyset$ for $i \neq j$, so this is in fact a probability measure.}
\end{defn}

We will not refer to the details of this path, merely  re-use bounds on path length, energy and congestion that were calculated in \cite{ChlebounSpm2018}.
It is also helpful to define the following naive path, also defined and analysed in \cite{ChlebounSpm2018}, which simply flips all $-1$ spins to $+1$ spins, in lexicographic order:

\begin{defn}[Naive Paths] \label{DefNaivePath}
For fixed $\sigma \in \Omega_{[1:L]^{2}}$, we define a ($\{0,1\}$-valued) probability measure $\flowna{\sigma}$ on $\Gamma_{\sigma, +}$ by giving an explicit algorithm for the path given weight 1 by this measure. Let $z^{(1)},  z^{(2)}, \ldots$ be all points of $[1:L]^{2}$ in lexicographic order.

\begin{enumerate}
\item Initialize by setting $\gamma = \{\sigma\}$ and $i=0$.
\item While $\fin{\gamma} \neq +$, do the following:
\begin{enumerate}
\item Set $i = i +1$.
\item If $\fin{\gamma}_{z^{(i)}} = -1$, do the following:
\begin{enumerate}
\item Set $\sigma = \fin{\gamma}$.
\item Set $\gamma = \gamma \conc ( \fin{\gamma}, \sigma^{z^{(i)}})$.
\end{enumerate}
\end{enumerate}
\item Return the path $\gamma$.
\end{enumerate}
\end{defn}
This path is rather inefficient in terms of energy, but is very short.

\subsection{Path Construction: New Intermediate Elements} \label{SubsecPathNew}

In this section, we construct a new ``intermediate" type of path and then put together our building blocks to define a path measure. Before giving the new definitions, we give some quick heuristics. Recall that we will be using all of our canonical paths to bound terms of the following form, appearing in Equation \eqref{IneqDefCanPathObject}:
\be
\mathcal{A}(k) \equiv 2\, \max_{e \in \edges} \sum_{\h \in \Omega} \sum_{\gamma \ni e} F^{(k)}_{\h}(\gamma) \, | \gamma | \, \frac{\mu(\h)}{\mu(e_{-}) K(e_{-},e_{+})}.
\ee

In the literature on canonical path arguments, people often describe the term $\sum_{\h \in \Omega} \sum_{\gamma \ni e} F^{(k)}_{\h}(\gamma)$ as the \textit{congestion}, the term $ | \gamma |$ as the \textit{path length}, and the term $\frac{\mu(\h)}{\mu(e_{-}) K(e_{-},e_{+})}$ as the \textit{energy.} We make some rough observations about these terms for a random path started at $\sigma \in S_{k}$:\footnote{The following describes only the ``worst-case" edges, and ignores polynomial factors of $\beta$. }

\begin{itemize}
\item If we sample $i \sim \mathrm{Unif}([1:n(\sigma)])$, $R \sim \mathrm{Unif}(F(\sigma,i))$, then the short (random) path $\gamma_{\sigma,R}$ will typically have congestion $\frac{L}{|D(\sigma)|}$, path length $L^{2}$, and energy $e^{2 \beta}$. In addition, the \textit{final element} of this path will have energy at most $e^{-2 \beta}$ times that of the \textit{first element.}
\item If we follow the naive path in Definition \ref{DefNaivePath}, the congestion and energy will both typically be $e^{\Theta(L)}$, but the \textit{total} path length will be only $L^{2}$.
\end{itemize}

Using only the first type of path gives a worst-case congestion of $\frac{L}{|D(\sigma)|}$, worst-case path length of $L^{4}$ (since there may be $\Theta(L^{2})$ rectangles to remove, each requiring a path of length $\Theta(L^{2})$), and worst-case maximum energy $e^{2 \beta}$. In \cite{ChlebounSpm2018} we use the first type of path only to clean up $\Theta(   \beta L)$ defects, achieving an energy surplus of $e^{\Theta( \beta  L)}$ by the end; we could then append the naive path, ensuring that the total path length remained $O(\beta L^{3})$.

In this section, we introduce a new type of path that will go in-between these two. It will have worst-case congestion of $e^{O(\beta)}$, worst-case path length of $O(\frac{L^{2}}{\sqrt{k}})$, and worst-case maximum energy of  $e^{O(\beta)}$. We will use the first path to remove $O(\beta^{2})$ defects, the new intermediate path to remove the next $O(\beta L)$, and the final naive path to remove the remainder. The result will have the same worst-case congestion and energy as in \cite{ChlebounSpm2018}, but the worst-case bound on the total path length will improve from $O(\beta L^{3})$ to $O(L^{2} \max(\beta^{2},\min(\sqrt{k}, \frac{L}{\sqrt{k}})))$.

The new path itself iteratively deletes the smallest-area region within some class of ``simple" regions. To be more precise:

\begin{defn} [Mid Path] \label{DefNewPath}
Fix $\sigma \in \config$. Define the pair $x,y$ by \[(x,y) = \argmin_{(a,b) \in D(\sigma)^2}\|a-b\|_1 \quad \textrm{with $x<y$ in lexicographic order;}\]
if there is a tie among several pairs, we break ties by comparing their first and then second elements in lexicographic order.


Denote by $R_{1}$ the smallest  rectangle in the family of rectangles with corners $x$, $(y_1,x_2)$, and $v(x,y)$, where $v(x,y)\in D(\s)$ is a defect vertically aligned with $x$ or $y$.
Let $R_{2}$ be the smallest rectangle in the family of rectangles with corners $y$, $(y_1,x_2)$, and $h(x,y)$, where $h(x,y) \in D(\s)$ is a defect horizontally aligned with $x$ or $y$ (see Fig. \ref{FigRelPostCong}).
Finally, define $\gamma(\s) = \gamma_{\sigma,R_{1}} \conc \gamma_{\eta,R_{2}}$ where $\h = \fin{\gamma_{\sigma,R_{1}}}$. If the defects at $x$ and $y$ are in the same row (respectively the same column) then $\cB_-(R_2)$ (respectively $\cB_-(R_1)$) will be empty and $\gamma(\s) = \gamma_{\s,R_2}$ (respectively $\gamma(\s) = \g_{\s,R_1}$).

For fixed $\sigma \in \Omega_{[1:L]^{2}}$, we define a ($\{0,1\}$-valued) probability measure $\flownew{\sigma}$ on $\Gamma_{\sigma, +}$ to be the measure concentrated on $\gamma(\s)$.
\end{defn}

We  now give bounds on path length and energy of these mid paths:

\begin{lemma}[Path-Length for Mid Paths]\label{LemmaPathLength} 
Fix $\L = [1:L]^{2}$ and $\sigma \in \config$. Let $\gamma=\g(\s)$ be as in Definition \ref{DefNewPath}, then
\be
|\gamma | \leq  6 \frac{ L^{2}}{\sqrt{|D(\sigma)|}}.
\ee
\end{lemma}

\begin{proof}










By a straightforward pigeon-hole argument we have
\be 
\label{IneqL1SquareSize}
\min \{ \| a - b \|_{1} \, : \, a,b \in D(\sigma) \} \leq  2 \frac{L}{\sqrt{|D(\sigma)|}} +1 \leq 3 \frac{L}{\sqrt{|D(\sigma)|}}.
\ee
The path $\gamma$ involves flipping all spins in two rectangles, both of which have one direction of size at most $\min \{ \| a - b \|_{\infty} \, : \, a,b \in D(\sigma) \}$ and the other direction of size at most $L$. Thus,
\be
|\gamma| &\leq 2 L \min \{ \| a - b \|_{\infty} \, : \, a,b \in D(\sigma) \}\leq 6 \frac{ L^{2}}{\sqrt{|D(\sigma)|}}.
\ee

\end{proof}



The following is very conservative:

\begin{lemma} [Energy for Mid Paths]  \label{LemmaNRGSmallRegRem}

Fix $\L = [1:L]^{2}$ and $\sigma \in \config$. Let $\gamma = \gamma(\s)$ be as in Definition \ref{DefNewPath}, and let $e \in \gamma$. Then
\be
\frac{\pi(\s)}{\pi(e_-)\cL(e_-,e_+)} \leq e^{52 \beta}\,.
\ee

\end{lemma}

\begin{proof}
  Fix $\L = [1:L]^{2}$, $\sigma \in \config$, and an $e \in \gamma$ as in the statement.
  By definition, any configuration $e_- \in \gamma$ differs from $\sigma$ by flipping spins in at most three rectangles.
  Observe that the defect configurations $D(\s)$ and $D(e_-)$ can only differ at the plaquettes touching the corners of the rectangle, of which there are at most $48$.
  Finally we apply the trivial bound on the jump rates $\cL(e_-,e_+) \geq e^{-4\b}$.
\end{proof}

\begin{lemma} [Congestion for Mid Paths] \label{LemmaCongSmallRegRem}

Fix $\L = [1:L]^{2}$. For $\sigma \in \config$, denote by $\gamma(\sigma)$ the path constructed in Definition \ref{DefNewPath}. For an edge $e = (e_{-}, e_{+})$ with $e_{-}, e_{+} \in \config$, we have
\be
|\{ \sigma \, : \, e \in \gamma(\sigma) \}| \leq 4 e^{6 \beta}.
\ee

\end{lemma}

\begin{proof}
  Following the same argument as the previous lemma, $e_-$ and $\s$ differ by spins in either zero, one, two, or three rectangles. Notice that each rectangle is uniquelly specified by it's top right and bottom left corner of which there are at most $L^2$ choices for each. It follows that  $|\{ \sigma \, : \, e \in \gamma(\sigma) \}| \leq L^{12}+L^{8}+L^4 + 1 \leq 4 e^{6\b}$.
\end{proof}
\begin{figure}[t]
	\centering
	\includegraphics[width=0.5\textwidth]{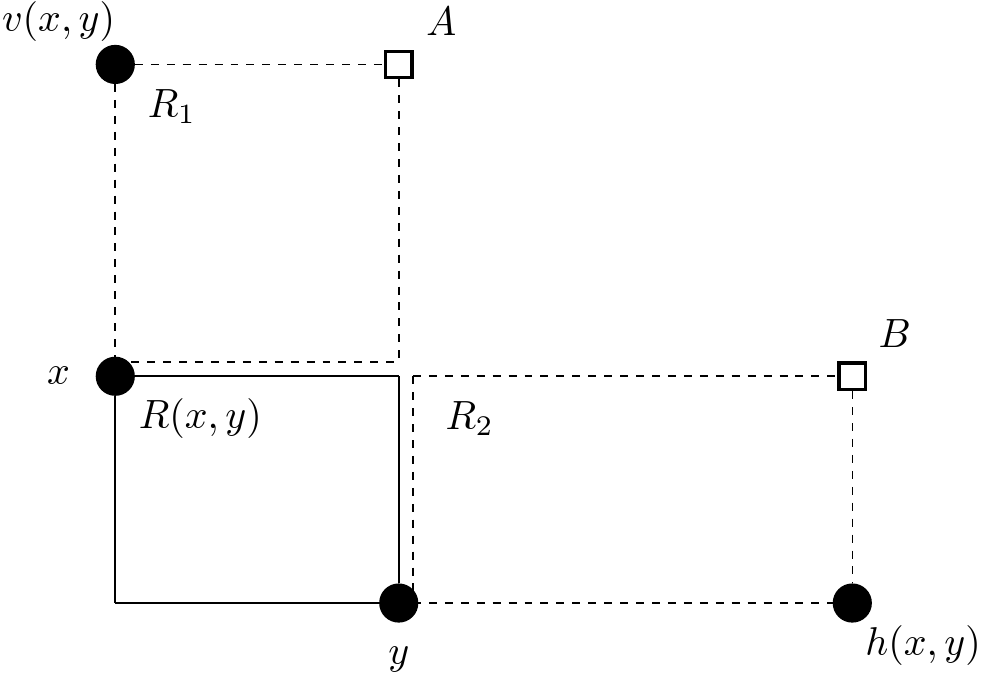}
	\caption{\label{FigRelPostCong} An example of the rectangles $R_1$, and $R_2$ as in Definition \ref{DefNewPath}, as well as the associated positions of the closest vertical and horizontal defects, $v(x,y)$ and $h(x,y)$ respectively. The points $A$ and $B$ are defined in the proof of Lemma \ref{LemmaAncestryLongNewPathBound}.}
\end{figure}
We now analyse the congestion associated with long, complete, parts of the full mid paths constructed in Definition \ref{DefNewPath}.
We do this by bounding the number of compatible paths associated with $\flownew{\cdot}$ which end at a certain configuration.
\begin{defn} [Compatible Mid Paths] \label{DefAncNonNew}
	
	Fix $\eta \in \Omega_{[1:L]^{2}}$ and $k \in 2\N$. We now define the set $\longnewanc{k}{\eta} \subset \Gamma$ of \textit{long compatible new paths} associated with $\eta$ as follows. Say $\gamma \in \longnewanc{k}{\eta}$ if it satisfies all of the following:
	
	\begin{enumerate}
		\item $\fin{\gamma} = \eta$, and
		\item $|D(\init{\gamma})| = |D(\eta)| + k$, and
		\item There exists $\sigma \in \Omega_{[1:L]^{2}}$ such that $\flownew{\sigma}(\gamma) = 1$.
	\end{enumerate}
	
\end{defn}

\begin{lemma} [Number of Compatible Mid Paths] \label{LemmaAncestryLongNewPathBound}
	
	For fixed $\eta \in \Omega_{[1:L]^{2}}$, we have the inequalities
	\be 
	|\longnewanc{2}{\eta}| &\lesssim L^{3}\,, \label{IneqAncestryLongNewPathMainBound1}\\
	|\longnewanc{4}{\eta}| &\lesssim L^{4} \quad \textrm{and, }\label{IneqAncestryLongNewPathMainBound2}\\
        |\longnewanc{k}{\eta}| &= 0 \quad \textrm{for} \quad k > 4\,.\nonumber
	\ee
\end{lemma}

\begin{proof}
	
 Fix $\eta \in \config$. We bound $|\longnewanc{k}{\eta}|$ by  parameterizing the collection of configurations $\sigma$ such that $\g(\s) \in\longnewanc{k}{\eta}$. We first study the bound when $k=2$ and then $k=4$. We break both these into two cases, according to whether the elements $x,y$ in Definition \ref{DefNewPath} are both in distinct rows \textit{and} in distinct columns, or not. 

\textbf{Case 1 ($k = 2$): distinct row and distinct column.} In this case $|D(\sigma) \Delta D(\eta)| = 6$, since the rectangles $R_1$ and $R_2$ in Definition \ref{DefNewPath} share exactly one vertex, $\cB_-(R_1)$ and $\cB_-(R_2)$ are both non-empty, and after flipping the spins in $\cB_-(R_1)$ and $\cB_-(R_2)$ the defect configuration will have flipped at all six vertices of $R_{1},R_{2}$ that are not shared by both rectangles.
	We label the defects in $D(\sigma)$ as in Definition \ref{DefNewPath}, and label the remaining two elements of $D(\s)\Delta D(\h)$ by $A,B$ (see Figure \ref{FigRelPostCong}). 
	We consider pairs $\sigma,\eta$ with defects in the relative position illustrated in Figure \ref{FigRelPostCong} (by ``relative position," we mean pairs $\sigma,\eta$ for which $x$ is to the left of $y$, $v(x,y)$ is above $x$, and $h(x,y)$ is right of $y$).

		For fixed $\eta$, we will parameterize the collection of such configurations $\sigma$ in terms of a set with $O(L^{3})$ elements. Since $k=2$ we know $A,B \in D(\eta)$; Clearly there are $O(|D(\eta)|)$ choices for the position $A$. Next, we choose $y$; since $y$ must be in the same column as $A$, there are $O(L)$ choices once we have chosen $A$. 
	Next, we choose $x$; since $x,y$ were at minimal distance, by Inequality \eqref{IneqL1SquareSize} there are $O(\frac{L^{2}}{|D(\eta)|})$ choices for $x$ once we have chosen $A$ and $y$. 
	Next, we choose $B$; by our construction, $B$ must be the first defect in $\eta$ to the right of $x$ and $y$ in the same row as $x$, so there is one choice for $B$ once we have chosen $A,y$, and $x$. Finally, $h(x,y)$ and $v(x,y)$ are determined by $x,y,A$ and $B$. Multiplying out the choices we have made, we see that there were $O(L^{3})$ options.
	
	\textbf{Case 2 ($k = 2$): $x,y$ are in either the same row or the same column.}
	In this case either $\cB_-(R_1)$ or $\cB_-(R_2)$ is empty. Without loss of generality suppose $\cB_-(R_2)= \emptyset$, then $|D(\sigma) \Delta D(\eta)| = 4$ and the point $B$ and $h(x,y)$ are coincident, otherwise the proof follows exactly as above.
	
	Finally, to this point we have considered only relative positions of the form shown in Figure \ref{FigRelPostCong}. It is clear that there are only $O(1)$ possible relative positions on a collection of $6$ points, and including the remaining positions would change this final bound by only a constant factor. This completes the proof for $k=2$ in both cases. The proof for $k=4$ is similar, and we mimic the notation above.

\textbf{Case 1 ($k = 4$): distinct row and distinct column.} 
 We show by contradiction that this case does not occur. Assume for contradiction that $|D(\sigma) \Delta D(\eta)| = 6$. 
	In this case there must be a defect at either $A$ or $B$ in $\s$, but this contradicts the minimality of $\|x-y\|_1$ over defect positions $x,y$.

	\textbf{Case 2 ($k = 4$): $x,y$ are in either the same row or the same column.}
	As for $k=2$ Case 2, either $\cB_-(R_1)$ or $\cB_-(R_2)$ is empty.
	Therefore to reconstruct $\s$ we only need to specify one rectangle, and this can be done in at most $L^4$ ways.

	By the same argument as \textbf{Case 1 ($k = 4$)}, we have $\longnewanc{j}{\eta}$ is empty for each $j > 4$. For parity concerns, it is empty for $j$ odd. This completes the proof.
	
	
\end{proof}

We then define the complete measure on paths that will be analyzed in the remainder of this paper:

\begin{defn} [Full Path] \label{DefFinalPath}

Fix $M_{1} = \lceil \beta^{2} \rceil$, $M_{2} = \lfloor \beta L \rfloor + M_{1}$.   For fixed $\sigma \in \Omega_{[1:L]^{2}}$, we define a  measure $\flowfull{\sigma}$ on $\Gamma_{\sigma,+}$ by giving an explicit algorithm for sampling from this measure:

\begin{enumerate}
\item Initialize by setting $\gamma = \{\sigma\}$ and $i=1$.
\item While $D(\fin{\gamma}) \neq \emptyset$, iteratively sample subpaths $\gamma_{1},\gamma_{2},\ldots$ according to the following loop.
\begin{enumerate}
\item If $i < M_{1}$, sample $\gamma_{i} \sim \flowtot{\fin{\gamma}}$, according to Definition \ref{DefPartialPath}.
\item If $M_{1} \leq i < M_{2}$, sample $\gamma_{i} \sim \flownew{\fin{\gamma}}$, according to Definition \ref{DefNewPath}.
\item If $i = M_{2}$, sample $\gamma_{i} \sim \flowna{\fin{\gamma}}$, according to Definition \ref{DefNaivePath}.
\item In all of these cases, set $\gamma = \gamma \conc \gamma_{i}$ and then $i = i+1$.
\end{enumerate}
\item Once $\gamma$ satisfies $D(\fin{\gamma}) = \emptyset$, return the path $\gamma$.
\end{enumerate}

When bounding the mixing time, it is useful to consider truncated paths as follows. Fix a truncation level $0 \leq k \leq L^{2}$. We define the measure $\flowtrunc{\sigma}{k}$ on $\Gamma$ by the following algorithm:
\begin{enumerate}
\item Sample the path $(\sigma^{(0)},\ldots,\sigma^{(m)}) \sim \flowfull{\sigma}$.
\item Let $i_{\min} = \min \{i \, : \, |D(\sigma^{(i)})| < k\}$.
\item Return the path $\gamma = (\sigma^{(0)},\ldots,\sigma^{(i_{\min})})$.
\end{enumerate}
\end{defn}

Having defined the full paths lengths we can now bound the total path length.
The following lemma is an improvement on the analogous bound in \cite{ChlebounSpm2018}, which had paths of length $O(L^{2} \min(k,\b L))$.

\begin{lemma} [Path Length] \label{LemmaPathLengthBound}
	Fix $\sigma \in \Omega_{[1:L]^{2}}$ with $|D(\s)| = k$ , and let $\gamma \in \Gamma_{\sigma,+}$ satisfy $\flowfull{\sigma}(\gamma) > 0$. Then the path length is bounded by 
	\be
	|\gamma| \lesssim  L^{2} \max(\min(\sqrt{k},\frac{\beta L}{\sqrt{k}}), \beta^{2}).
	\ee
	
\end{lemma}

\begin{proof}

Fix a path $\gamma$ with $\flowfull{\sigma}(\gamma) > 0$. Inspecting Definition \ref{DefFinalPath}, we can write $\gamma = \gamma_{1} \conc \gamma_{2} \conc \ldots \conc \gamma_{m}$, where each sub-path $\gamma_{i}$ was built in part \textbf{(2a)}, \textbf{(2b)} or \textbf{(2c)} of the definition. We now bound the number and length of sub-paths from each category: 

There are $\leq \beta^{2}$ paths from \textbf{(2a)}, each clearly of length $\leq L^{2}$; there are $\leq \beta L$ paths from \textbf{(2b)}, the $d$'th of which is of length $\leq \frac{ 6L^{2}}{\sqrt{k-4d - 4\beta^{2}}}$ by Lemma \ref{LemmaPathLength} (since each subpath $\g_i$ removes at most $4$ defects for $i<M_2$); there is $\leq 1$ path from \textbf{(2c)}, and it is clearly of length $\leq L^{2}$. Putting these three bounds together, the total path length is bounded by
\be \label{IneqPathLengthMotivation}
|\g| & \lesssim \beta^{2} L^{2} +  \sum_{d=\max\{0, k-\beta L\}}^{k} \frac{ L^{2}}{\sqrt{d}} + L^{2} \lesssim L^{2} \max(\min(\sqrt{k},\frac{\beta L}{\sqrt{k}}), \beta^{2}).
\ee
\end{proof}


\subsection{Final Bounds on Canonical Paths} \label{SecFinalBoundCanonical}

In this section, we put together our main bounds on the canonical paths studied in this paper. The calculation is somewhat lengthy, so we remind the reader of some conventions that are used throughout:

\begin{enumerate}
\item Recall the constants $M_{1} = \lceil \beta^{2} \rceil $ and $M_{2} = \lfloor \beta L \rfloor + M_{1}$, the number of steps taken before switching ``types" of paths in Definition \ref{DefFinalPath}.
\item For an initial configuration $\s \in \config$, we define $\ell(\s) = \max(\min(\sqrt{|D(\s)|}, \frac{\beta L}{\sqrt{|D(\s)|}}), \beta^{2})$.
\item  Observe that $F$ in Definition \ref{DefFinalPath} is in fact a measure on \textit{decompositions} of paths of the form
  \be \label{EqMainLemmaShortPathRepInt}
    \gamma_{1} \conc \gamma_{2} \conc \ldots \conc \gamma_m\,,
  \ee
not just a measure on paths themselves. 
Recall that $\gamma_i$ are of the form of described in Definition \ref{DefPartialPath} for $i < M_1$,  of the form of described in Definition \ref{DefNewPath} for $M_1 \leq i < M_2$, and if $m = M_2$ then $\gamma_{M_2}$ is of the form described in Definition \ref{DefNaivePath}. We will sum over all possible decompositions of the relevant paths.

\item The measures  $\{\flowtot{\sigma}\}$ are as in Definition \ref{DefPartialPath}, the measures  $\{\flowfull{\sigma}\}$, $\{\flowtrunc{\sigma}{k}\}$ are as in Definition \ref{DefFinalPath}, the measures $\{ \flownew{\sigma} \}$ are as in Definition \ref{DefNewPath}, and the measures $\{ \flowna{\sigma} \}$ are as in Definition \ref{DefNaivePath}. We note that $\flowna{\sigma}$ and $\flownew{\s}$ assign full mass to a single path, which we denote $\gamma_{\sigma,\mathrm{na}}$ and $\g_{\s,\mathrm{mid}}$ respectively.
\item For $\s \in \config$ and $e = (e_{-},e_{+})$, we define the energy of a path started from $\s$ on the edge $e$ by $Q(\s,e) = \pi(\s) /\big( \pi(e_{-}) \cL(e_{-},e_{+}) \big)$ .
\end{enumerate}

We fix an edge $e \in \edges$ and calculate a bound on the sum that appears in Equation \eqref{IneqDefCanPathObject} for some $k \in [0\!:\!L^{2}]$. Although in principle we could take advantage of the fact that paths from $S_{k}$ that are truncated when they enter $S_{k}^{c}$ are shorter than paths which go all the way from $S_{k}$ to $\{+\}$, we will not do so. However, we \textit{will} use the fact that all elements of a path from $S_{k}$ to $S_{k}^{c}$ have $\Omega(k)$ defects.

Recalling $M_{1} \approx \beta^{2}$ and $M_{2} \approx \beta L$, we have:

\be \label{IneqCanPathStartingCalc}
\sum_{\sigma \in S_{k}} \sum_{\gamma \ni e} \flowtrunc{\sigma}{k}(\gamma) | \gamma|  Q(\s,e) &\stackrel{Lemma \,\, \ref{LemmaPathLengthBound}}{\lesssim} \sum_{\sigma \in S_{k}} \sum_{\gamma \ni e} \flowtrunc{\sigma}{k}(\gamma) L^{2} \ell(\s) Q(\sigma,e) \\
&\leq   \sum_{m = 1}^{M_{2}} \sum_{j=1}^{m} \sum_{\substack{\gamma = \gamma_{1} \conc \ldots \conc \gamma_{m} \\ \init{\gamma} \in S_{k},\, \gamma_{j} \ni e}} \flowtrunc{\init{\gamma}}{k}(\gamma) L^{2} \ell(\init{\g}) Q(\init{\gamma},e) \\
&=   \sum_{m = 1}^{M_{2}} \sum_{j=1}^{\min(m,M_{1}-1)} \sum_{\substack{\gamma = \gamma_{1} \conc \ldots \conc \gamma_{m}\\ \init{\gamma} \in S_{k}, \, \gamma_{j} \ni e}} \flowtrunc{\init{\gamma}}{k}(\gamma) L^{2} \ell(\init{\g}) Q(\init{\gamma},e)  \\
&\qquad+  \sum_{m = M_1}^{M_{2}} \sum_{j=M_{1}}^{\min(m,M_{2}-1)} \sum_{\substack{\gamma = \gamma_{1} \conc \ldots \conc \gamma_{m},\\ \init{\gamma} \in S_{k},  \, \gamma_{j} \ni e}} \flowtrunc{\init{\gamma}}{k}(\gamma) L^{2} \ell(\init{\g}) Q(\init{\gamma},e)  \\
&\qquad+  \sum_{\substack{\gamma = \gamma_{1} \conc \ldots \conc \gamma_{M_{2}} \\ \init{\gamma} \in S_{k}, \, \gamma_{M_2} \ni e}} \flowtrunc{\init{\gamma}}{k}(\gamma) L^{2} \ell(\init{\g}) Q(\init{\gamma},e) \\
&\equiv  L^{2} S_{\mathrm{init}} +   L^{2}  S_{\mathrm{new}} +  L^{2}  S_{\mathrm{naive}}.
\ee

To bound the first term, $ L^{2} S_{\mathrm{init}}$, we use a calculation from Section 5.6 of our previous paper \cite{ChlebounSpm2018}, with one small modification. Although the paths studied in \cite{ChlebounSpm2018} are not the same as those in this paper, the \textit{initial} components of the paths are the same. The analysis between (5.51)-(5.60)  in Section 5.6 of  \cite{ChlebounSpm2018} is only related to these initial components, and so can be re-used with only one change: the single reference to the path length bound  in Lemma 5.32 of \cite{ChlebounSpm2018} should be replaced by a reference to Lemma \ref{LemmaPathLengthBound} of this paper. Following (5.51)-(5.60) of \cite{ChlebounSpm2018} with this substitution gives: 
\be \label{QuotePreviousPathBounds2}
 L^{2}S_{\mathrm{init}}   \lesssim \beta^{6} e^{3.5 \beta} \max\Big\{\frac{1}{\sqrt{|D(e_{-})|}} \min\big(1, \frac{\beta L}{|D(e_{-})|}\big),\, \frac{\beta^{2}}{|D(e_-)|}\Big\}.
\ee

We now bound $S_{\mathrm{new}}$. The argument will be quite similar to Section 5.6 of \cite{ChlebounSpm2018}. As the subscripts in the following sums will become somewhat complicated, we introduce some short-term notation to reduce the visual clutter. 
\be
\Gamma^{(i)}(m) &= \{ \gamma_{1} \conc \ldots \conc \gamma_{m} \in \Gamma \, : \, \, \gamma_{i} \ni e, \, \init{\gamma_{1}} \in S_{k} \}\,, \\
\Gamma(m) &= \{ \gamma_{1} \conc \ldots \conc \gamma_{m} \in \Gamma \, : \, \, \gamma_{m} \ni e, \, \init{\gamma_{1}} \in S_{k} \}\,. 
\ee
As explained in the introduction to this subsection, we view the elements of \textit{e.g.} $\Gamma(m)$ as \textit{paths with a specified decomposition of the form \eqref{EqMainLemmaShortPathRepInt}}, not just paths. Thus, a ``single path" may appear more than once in \textit{e.g.} $\Gamma(m)$ if it has several decompositions of the form \eqref{EqMainLemmaShortPathRepInt}. Also, we have omitted the dependence on many variables (e.g. the edge $e$ and starting level $k$) that are constant throughout the following calculation.

We then compute:
\begin{align}
\label{eq:Snewfirst}
S_{\mathrm{new}} &=    \sum_{m = M_1}^{M_{2}} \sum_{i=M_{1}}^{\min(m,M_{2}-1)} \sum_{\substack{\gamma \in\Gamma^{(i)}(m)}} \flowtrunc{\init{\gamma}}{k}(\gamma)\ell(\init{\g}) Q(\init{\gamma},e) \nonumber\\
&= \sum_{m = M_{1}}^{M_{2}-1}  \sum_{\gamma \in \Gamma(m)}  \flowtrunc{\init{\gamma}}{k}(\gamma)\ell(\init{\g}) Q(\init{\gamma},e) \nonumber \\
&= \sum_{m = M_{1}}^{M_{2}-1}  \underbrace{\sum_{\gamma \in \Gamma(m)} \left( \prod_{j=1}^{M_{1}-1} \flowtot{\init{\gamma_{j}}}(\gamma_{j}) \right) \left(\prod_{j=M_{1}}^{m} \flownew{\init{\gamma_{j}}}(\gamma_{j}) \right) \ell(\init{\g}) Q(\init{\gamma},e)}_{\equiv \cS(m,e)}\,, 
\end{align}
where in the second line we took the marginal distribution of the first $i$ segments of the path $\gamma$ that is being summed over (we can ``integrate out" the remaining segments from index $i+1$ to $m$, as they are not related to the other terms in the sum).

We now ``factor" the paths in $\Gamma(m)$ into pieces corresponding to a ``long initial" section composed of $M_1-1$ partial random paths (see Def. \ref{DefPartialPath}), a ``long middle" section composed of $m-M_1-1$ mid paths (see Def. \ref{DefNewPath}), and a ``final segment" corresponding to a mid path that contains the edge $e$.  We associate three sets to this path decomposition as follows, as illustrated in Figure \ref{FigMidPathAnalysisDecomp}.
\begin{figure}[ht]
\centering
\includegraphics[width=0.3\textwidth]{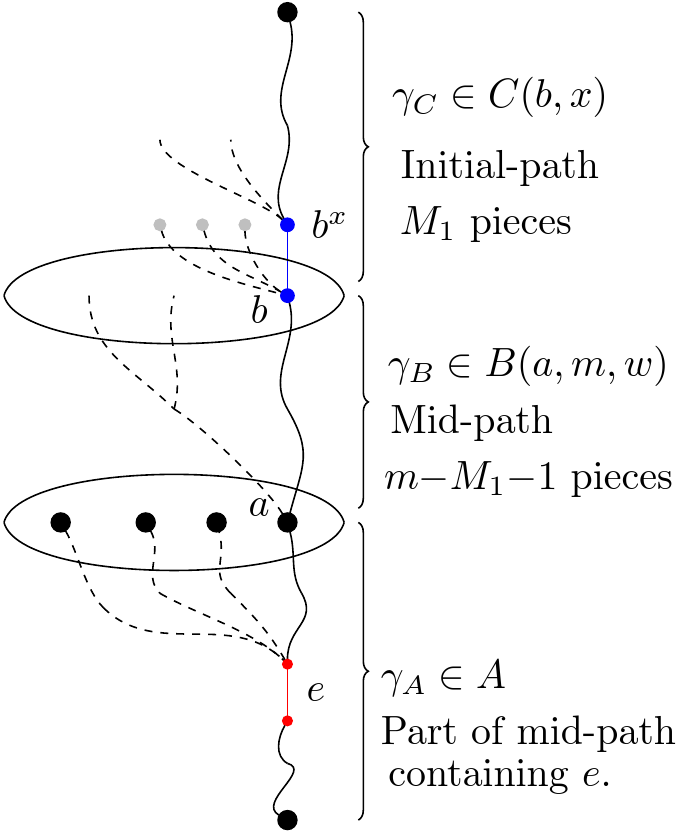}\qquad \qquad \qquad \includegraphics[width=0.3\textwidth]{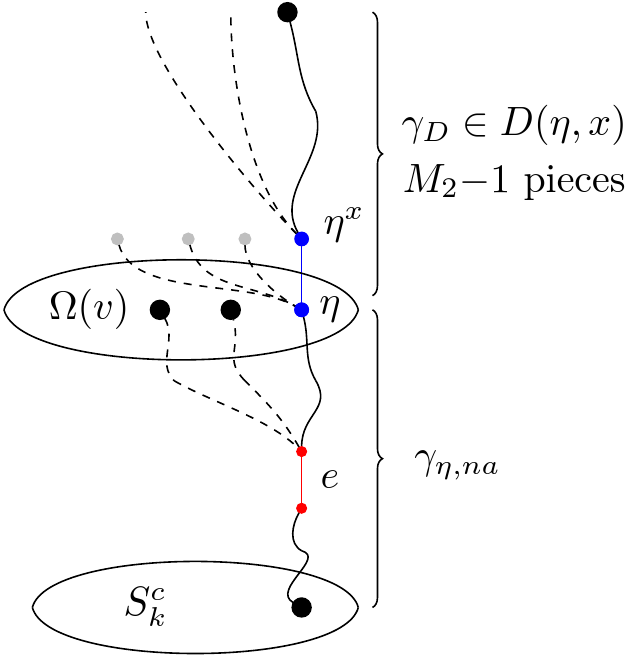}
\caption{ A cartoon of the path decomposition used. Left: for $S_{\textrm{new}}$, when the edge $e$ is contained in a a mid-path. Right: for $S_{\textrm{naive}}$, when the edge $e$ is contained in the final naive part of the path.\label{FigMidPathAnalysisDecomp}}
\end{figure}

In the following definitions, we always think of  $2w$ as the number of ``excess" defects removed in going from the end of the initial random path's to $\init{\g_m}$ beyond the minimal number $2(m-M_1)$. 
\be
A &= \{ \gamma \, : \, \gamma \ni e, \, \flownew{\init{\gamma}}(\gamma) = 1 \}, \\
B(a,m,w) &= \{ \gamma_{B,1} \conc \ldots \conc \gamma_{B,m-M_{1}} \, : \, \fin{\gamma_{B,m-M_{1}}} = a, \, \forall \, j \in [1:m-M_{1}-1] \ \flownew{\fin{\gamma_{j,B}}}(\gamma_{j+1,B}) = 1, \\
& \qquad \quad\, |D(\init{\gamma_{1,B}})| = |D(a)| + 2 (m-M_{1}+w) \}, \\
C(b) &= \{ \gamma_{C,1} \conc \ldots \conc \gamma_{C,M_{1}-1} \, : \, \init{\g_{C,1}} \in S_k,  \,\fin{\gamma_{C,M_{1}-1}} = b, \, \forall \, j \in [1:M_{1}-2] \ \flowtyp{\fin{\gamma_{C,j}}}(\gamma_{C,j+1}) >0 \}, \\
C(b,x) &= \{ \gamma_{C,1} \conc \ldots \conc \gamma_{C,M_1-1} \, : \, \init{\g_{C,1}} \in S_k,  \, \gamma_{C,M_{1}-1} \ni (b^x,b),\ \forall \, j \in [1:M_{1}-2] \ \flowtyp{\fin{\gamma_{C,j}}}(\gamma_{C,j+1}) >0 \},\\
h(m) &= \min(\sqrt{|D(e_-)|}+2\sqrt{m},\frac{\beta L}{\sqrt{|D(e_-)|}})+\b^2\,.
\ee
We note that $C(b) \subset \cup_{x \in \L} C(b,x)$ for all $b \in \config$.
We now decompose \eqref{eq:Snewfirst} and apply the bound $\ell(\init{\g}) \leq h(m)$, which follows from the bounds $|D(e_-)| \leq |D(\init{\g})|$ and $|D(e_-)| +4m \geq |D(\init{\g})|$.
We then write:
\be
S_{\mathrm{new}} &\leq \sum_{m=M_{1}}^{M_{2}-1} \sum_{\gamma_{A} \in A} \sum_{w=0}^{m-M_{1}} \sum_{\gamma_{B} \in B(\init{\gamma_{A}},m,w)}  \left(\prod_{j=1}^{m-M_{1}} \flownew{\init{\gamma_{B,j}}}(\gamma_{B,j}) \right) \cdot \\ &\qquad\cdot \sum_{\gamma_{C}  \in C(\init{\gamma_{B,1}})} \left( \prod_{j=1}^{M_{1}-1} \flowtot{\init{\gamma_{C,j}}}(\gamma_{C,j}) \right) h(m) Q(\init{\gamma_{C}},e) \\
&\leq  \sum_{m=M_{1}}^{M_{2}-1} \sum_{\gamma_{A} \in A} \sum_{w=0}^{m-M_{1}} \sum_{\gamma_{B} \in B(\init{\gamma_{A}},m,w)}  \left(\prod_{j=1}^{m-M_{1}} \flownew{\init{\gamma_{B,j}}}(\gamma_{B,j}) \right) \cdot \\ &\qquad\cdot \sum_{x \in \L} \sum_{\gamma_{C}  \in C(\init{\gamma_{B,1}},x)} \left( \prod_{j=1}^{M_{1}-1} \flowtot{\init{\gamma_{C,j}}}(\gamma_{C,j}) \right) h(m) Q(\init{\g_C},e)\,.
\ee
Let $f_x^{\g_B} =\big((\init{\g_{B,1}})^x,\init{\g_{B,1}}\big) \in \edges$, then
\begin{align*}
  \frac{Q(\init{\g_C},e)}{Q(\init{\g_{C}},f_x^{\g_B})} =\frac{\pi\big((\init{\g_{B,1}})^x\big)\cL(f_x^{\g_B})}{\p(e_-)\cL(e)}  \leq Q(\init{\gamma_{B,1}},e) \,,
\end{align*}
where in the inequality we used the fact  $\pi\big((\init{\g_{B,1}})^x\big)\cL\big((\init{\g_{B,1}})^x,\init{\g_{B,1}}\big) \leq \pi\big(\init{\g_{B,1}}\big)$; this is immediate from the facts that the process is reversible and all rates are bounded by $1$.
Substituting this into the previous inequality and continuing,

\begin{align}
 \label{IneqSNewFactorSInit}
S_{\mathrm{new}} & \lesssim \sum_{m=M_{1}}^{M_{2}-1} \sum_{\gamma_{A} \in A} \sum_{w=0}^{m-M_{1}} \sum_{\gamma_{B} \in B(\init{\gamma_{A}},m,w)}  \left(\prod_{j=1}^{m-M_{1}} \flownew{\init{\gamma_{j,B}}}(\gamma_{B,j}) \right)Q(\init{\gamma_{B,1}},e) \cdot \\* &\qquad\cdot \sum_{x \in \L} \underbrace{  \sum_{\gamma_{C}  \in C(\init{\gamma_{1,B}},x)} \left( \prod_{j=1}^{M_{1}-1} \flowtot{\init{\gamma_{C,j}}}(\gamma_{C,j}) \right) h(m) Q(\init{\g_{C}},f_x^{\g_B})}_{S'(M_{1}-1,f_x^{\g_B})}\,.
\end{align}
We now recognize that the underlined term, $S'(M_{1}-1,f_x^{\g_B})$, is nearly identical to the term $S(M_{1}-1,f_x^{\g_B})$ defined in Equation (5.52) of \cite{ChlebounSpm2018}, except we have to replace the term $\min(|D(\init{\g_C})|,\b L)$ in the previous paper with $h(m)$ here.
We apply Inequality (5.59) of \cite{ChlebounSpm2018}, together with the fact $|D\big((\init{\g_{B,1}})^x\big)| \geq 1$,  to find:
\be
S'(M_{1}-1,f_x^{\g_B}) \lesssim \beta^{5} e^{2.5 \beta} \beta^{-M_{1}-2} h(m).
\ee

Applying this bound and the trivial inequality $h(m) \leq \beta L$, we have:
\begin{align} 
\label{QuotePreviousPathBoundsLongNew}
S_{\mathrm{new}} &\lesssim \beta^{6} e^{4 \beta} \beta^{-M_{1}-2}  \sum_{m=M_{1}}^{M_{2}-1} \sum_{\gamma_{A} \in A} \sum_{w=0}^{m-M_{1}}  \sum_{\gamma_{B} \in B(\init{\gamma_{A}},m,w)}  \left(\prod_{j=1}^{m-M_1} \flownew{\init{\gamma_{B,j}}}(\gamma_{B,j}) \right) Q(\init{\gamma_{B,1}},e)\nonumber\\
&\stackrel{\textrm{Lemma  \ref{LemmaAncestryLongNewPathBound}}}{\lesssim} \beta^{6} e^{4 \beta} \beta^{-M_{1}-2}  \sum_{m=M_{1}}^{M_{2}-1} \sum_{\gamma_{A} \in A} \sum_{w=0}^{m-M_{1}}  \binom{m - M_{1}}{w} L^{4w} L^{3(m-M_{1}-w)} e^{-2( m-M_{1} + w)\beta} Q(\init{\gamma_{A}},e)\nonumber\\
&\stackrel{m \lesssim \beta L}{\lesssim} \beta^{6.5} e^{4 \beta} \beta^{-M_{1}-2}  \sum_{m=M_{1}}^{M_{2}-1} \sum_{\gamma_{A} \in A} e^{-0.5(m-M_{1})\b} Q(\init{\gamma_{A}},e) \lesssim \beta^{5} e^{4 \beta} \beta^{-M_{1}-2}  \sum_{\gamma_{A} \in A}  Q(\init{\gamma_{A}},e) \nonumber\\
&\stackrel{\textrm{Lemmas \ref{LemmaNRGSmallRegRem},  \ref{LemmaCongSmallRegRem}}}{\lesssim} \beta^{6.5} e^{62 \beta} \beta^{-M_{1}-2}  \lesssim e^{-\beta^{1.5}},
\end{align}
where the last line is an extremely lose bound using the fact that $M_{1} \gtrsim \beta^{2}$.  We note that this term is negligible compared to $S_{init}$.

We now bound the remaining term, $S_{\textrm{naive}}$. Let 
\begin{align*}
\O(v) &= \{\h \in \config \,:\, |D(\h)| = |D(e_-)| + 2v,\ \gamma_{\h,\textrm{na}} \ni e\}\,\\
D(\h,x) &= \{ \g_1\conc\ldots \conc \g_{M_2-1}\in \G\,:\, \init{\g_1}\in S_k,\ \g_{M_2-1} \ni (\h^x,\h),\ \flowtrunc{\init{\g_1}}{k}(\g_2\conc \ldots \conc \g_{M_2 -1}\conc \g_{\h,\textrm{na}}) > 0\}\,.
\end{align*}
By a similar decomposition as above 
\begin{align*}
S_{naive} &= \sum_{\substack{\gamma = \gamma_{1} \conc \ldots \conc \gamma_{M_{2}} \\ \init{\gamma} \in S_{k}, \, \gamma_{M_2} \ni e}} \flowtrunc{\init{\gamma}}{k}(\gamma) \ell(\init{\g})Q(\init{\gamma},e) \\
&\leq \sum_{2v \geq -(L+1)}\sum_{\h \in \O(v)} Q(\h,e)  \sum_{x\in\L}\underbrace{\sum_{\g_D \in D(\h,x)} \flowtrunc{\init{\gamma_D}}{k}(\gamma_D) \ell(\init{\g_D})Q(\init{\gamma_D},(\h^x,\h))}_{ = \cS(M_2-1,(\h^x,\h)) }\,,
\end{align*}
where we can identifying the under-braced term with $\cS(M_2-1,(\h^x,\h))$ defined by the under-braced term in Eq. \eqref{eq:Snewfirst}.
It then follows from the analysis above that

\begin{align*}
S_{naive} &\lesssim  \beta^{6.5} e^{62 \beta} \beta^{-M_{1}-2}  e^{-0.5(M_2-M_{1}-1)\b} L^2\sum_{2v \geq -(L+1)}\sum_{\h \in \O(v)} Q(\h,e)\,\\
 &\lesssim \beta^{-M_1 + 5.5} e^{-40 L \beta} \underbrace{ \sum_{2v \geq -(L+1)}\sum_{\h \in \O(v)} e^{-2v\b} }\ .
\end{align*}
The underbraced term is bounded above in Ineq. (5.63) of \cite{ChlebounSpm2018}. Applying this bound,
\begin{align}
\label{QuotePreviousPathBoundsLongNaive}
S_{naive} &\lesssim \beta^{-M_1 + 5.5}2^{2L} e^{-30 L \beta}\,,
\end{align}
which is extremely small compared to our bound on $S_{\mathrm{init}}$ for $\beta$ sufficiently large.

Summarizing the calculations in this section, we have by Inequalities \eqref{IneqCanPathStartingCalc}, \eqref{QuotePreviousPathBounds2}, \eqref{QuotePreviousPathBoundsLongNew} and \eqref{QuotePreviousPathBoundsLongNaive},
\be \label{IneqMainCanonicalPathConclusion}
\sum_{\sigma \in \config} & \sum_{\gamma \ni e} \flowfull{\sigma}(\gamma) | \gamma| \frac{\pi(\sigma)}{\pi(e_{-}) \cL(e_{-},e_{+})} \lesssim L^{2}(S_{\mathrm{init}} + S_{\mathrm{new}} +  S_{\mathrm{naive}}) \\
&\lesssim \beta^{6} e^{3.5 \beta} \max\Big\{\frac{1}{\sqrt{|D(e_{-})|}} \min\big(1, \frac{\beta L}{|D(e_{-})|}\big),\, \frac{\beta^{2}}{|D(e_-)|}\Big\}.
\ee
for any edge $e$.

\section{Analysis of All-Plus Boundary Condition: Proof of Theorem \ref{ThmMainResPlus}}  \label{SecAllPlusRes}

For $r  >0 $, and $k(r)$ as defined in Equation \eqref{DefHighDensityRound}, and $\mathcal{A}(k(r))$ defined as in Equation \eqref{IneqDefCanPathObject} for the paths analyzed in this section, Inequality \eqref{IneqMainCanonicalPathConclusion} gives:  
\be \label{IneqMainSpecProfConclusion}
\mathcal{A}(k(r)) \lesssim\begin{cases}
	\beta^{8} \, e^{3.5 \beta}, & r \geq e^{-\b^{5}}\,, \\
	\beta^{6.5} \, e^{3.5 \beta} \frac{1}{\sqrt{\log(1/r)}}, & e^{- \beta L } < r < e^{-\b^{5}}, \\
	\beta^{8.5} \, e^{4 \beta} \frac{1}{(\log(1/r))^{1.5}}, & 0 < r < e^{-  \beta L },
\end{cases}
\ee
where we recall from Lemma \ref{lem:dom} that $\pi(+) \approx 1$. Applying this bound with Theorem \ref{ThmMainSpectralProfileBound}, we conclude 
\be
T_{\rm mix}^+( L_c ) &\stackrel{\eqref{IneqMainSpectralProfileBound}}{\leq} \int_{4 \, \min_{\sigma} \pi(\sigma)}^{16} \frac{2 \mathcal{A}(k(x))}{x} dx \\
&\stackrel{\eqref{IneqMainSpecProfConclusion} }{\lesssim}  \beta^{8.5} e^{4 \beta} \int_{4  e^{-L^{2} \beta}}^{e^{-L \beta}} \frac{1}{x \log(\frac{1}{x})^{1.5}} dx +  \beta^{6.5} e^{3.5 \beta} \int_{e^{-L \beta}}^{e^{-\b^{5}}} \frac{1}{x \sqrt{\log(\frac{1}{x})}} dx+\beta^{8} e^{3.5 \beta} \int_{e^{-\beta^{5}}}^{16} \frac{1}{x} dx  \\
&\lesssim \beta^{8} e^{3.75\b},
\ee
completing the proof of the upper bound in Inequality \eqref{IneqPlusMix}.


\section{Analysis of Periodic Boundary Condition: Proof of Theorem \ref{ThmMainResPer}} \label{SecPerBoundRes}

We prove Theorem \ref{ThmMainResPer}. Throughout this section, we often reserve subscripts for a time index and use the ``bracket" notation $x[i] \equiv x_{i}$ to indicate an element of a vector $x = (x_{1},\ldots,x_{k})$. 
In this section, we have many explicit probabilistic calculations related to the Markov processes $\{X_{t}\}_{t \geq 0}$ of the spin dynamics running according to the generator \eqref{eq:gen}, with periodic boundary conditions, and with starting points $X_{0} =\s$. 
When we wish to emphasize the starting point of a process in such a calculation, we use subscripts, as in \textit{e.g.} $\P_{\s}[X_{t} \in S]$, $\E_{\s}[f(X_{t})]$.

\subsection{Initial Notation}

We begin by giving the heuristic that guides our proof of this inequality.  Throughout this section, we denote by $\{ \hat{X}_{t}\}_{t \geq 0}$  the usual \textit{trace} of the process $\{X_{t}\}_{t \geq 0}$ on the set $\cG$ of ground states, and denote by $Q_{\cG}$ the generator of this process (see \textit{e.g.} Section 2.2 of \cite{landim2018metastable} for the definition of the trace and some of its basic properties). For the coupled pair of processes $\{X_{t}\}_{t \geq 0}$, $\{\hat{X}_{t}\}_{t \geq 0}$ and $T \geq 0$, denote by $s_{\cG}(T)$ the Lebesgue measure of the set of times $\{0 \leq t \leq T \, : \, X_{t} \in \cG \}$ up to $T$ in which $\{X_{t}\}_{t \geq 0}$ is within $\cG$. Note that $\hat{X}_{s_{\cG}(t)} = X_{t}$ for almost all values of $t$ for which $X_{t} \in \cG$. Finally, denote by $P_{\cG}^{t}$ the law of the Markov chain associated with $Q_{\cG}$ at time $t$.

Roughly speaking, we expect the mixing time of the dynamics of the spin model $\{X_{t}\}_{t \geq 0}$ to be determined by the dynamics of this trace process. Furthermore, the trace behaves very much like a simple random walk on the hypercube $\{-1,+1\}^{2L-1}$. This can be made precise in the following way: there is a natural bijection $w \, : \, \cG \mapsto \{-1,+1\}^{2L-1}$  (see Definition \ref{DefBijectionPeriodic} below) under which the dynamics of the trace walk $\{w(\hat{X}_{t}) \}_{t \geq 0}$ are a very small perturbation of the usual simple random walk on the hypercube $\{-1,+1\}^{2L-1}$ with generating set $\{(-1,1,1,\ldots,1),(1,-1,1,\ldots,1),\ldots,(1,1,1,\ldots,-1),(-1,-1,-1,\ldots,-1)\}$. It will turn out that this perturbation is small enough that $\{w(\hat{X}_{t}) \}_{t \geq 0}$ inherits cutoff from the random walk on the hypercube.

We now make this heuristic precise. We begin by recalling from \cite{ChlebounSpm2018} the main bijection between the set of ground states $\cG$ and the hypercube:

\begin{defn} [Bijection] \label{DefBijectionPeriodic}
 We define a bijection $w \, : \, \cG \mapsto \{-1,+1\}^{2L-1}$ between the collection of ground states and the $(2L-1)$-hypercube by the following formula for its inverse:
\be \label{EqCristinaBijection}
w^{-1}(v) [i,j] &= v[i] v[L+j], \qquad i \in [1:L], \, j \in [1:(L-1)]\,, \\
w^{-1}(v)[i,L] &= v[i], \qquad \qquad \quad \,  i \in [1:L].
\ee

For $v \in \{ -1, 1\}^{2L-1}$, define $|v| = \sum_{i=1}^{2L-1} \1_{v[i] = 1}$ to be the Hamming weight of $v$. We also denote by $\prec$ some fixed total order on $\{-1,+1\}^{2L-1}$ that extends the usual Hamming partial order.

\end{defn}

Our goal is to prove that cutoff occurs at time $\approx \beta e^{4 \beta}$. The main idea is to note that a chain $\{X_{t}\}_{t \geq 0}$  started at   $X_{0} \in \cG$ is unlikely to see any configurations with more than 4 defects over the relevant time-scale. Furthermore, \textit{most} configurations with 4 defects are quite unlikely: we typically only see configurations that are along  \textit{minimal-energy}, \textit{minimal-length} paths between elements of $\cG$.

We give a formal description of these configurations here. 
Some readers may prefer to skip ahead to Remark \ref{RemMinEnPath} and Figure \ref{FigDominantTransitionsBetweenGroundStates}, which  give essentially-complete descriptions with much lighter notation.

\begin{defn}[Minimal-Energy Paths Between Ground States] \label{DefMinEnMinLPaths}

For  $m \in [1:L]$, $k \in [1:L]$, and $\sigma, \eta \in \cG$ with $w(\sigma)$, $w(\eta)$ differing at the single index $\ell \in [1\!:\!L]$ and with $w(\sigma) \prec w(\eta)$, define the configuration $R(\sigma,\eta,m,k)$ by
\be
R(\sigma, \eta,m,k)[\ell,j] &= \eta[\ell,j], \qquad j \in [m:(m+k-1)] \\
R(\sigma, \eta, m,k) [i,j] &= \sigma[i,j], \qquad \text{all other entries},
\ee
where in this case the ``interval" $[m\! :\! (m+k-1)]$ is defined modulo $L$ (so that, if $L = 9$, we have $[7\!:\!2] = \{7,8,9,1,2\}$).

Similarly, for $\sigma,\eta$ with $w(\sigma)$, $w(\eta)$ differing at the single index $\ell \in [(L+1):(2L-1)]$ and with $w(\sigma) \prec w(\eta)$, define the configuration $R(\sigma,\eta,k,m)$ by
\be
R(\sigma, \eta,m,k)[i,\ell] &= \eta[i,\ell], \qquad i \in [m:(m+k-1)] \\
R(\sigma, \eta, m,k) [i,j] &= \sigma[i,j], \qquad \text{all other entries.}
\ee
Finally, for $\sigma, \eta$ with $w(\sigma) = - w(\eta)$ and $w(\sigma) \prec w(\eta)$, define the configuration $R(\sigma, \eta, m,k)$ by
\be
R(\sigma, \eta,m,k)[i,L] &= \eta[i,L], \qquad i \in [m:(m+k-1)] \\
R(\sigma, \eta, m,k) [i,j] &= \sigma[i,j], \qquad \text{all other entries.}
\ee
For convenience,  when $w(\s) \prec w(\h)$ we let $R(\s,\h,m,0)=\s$, and when $w(\eta) \prec w(\sigma)$ we define $R(\sigma,\eta,m,k) \equiv \emptyset$ for each $k\in [0:L]$.

We note that, if you restrict the range of $k$ to the set $[2:L-2]$, then $R$ is an injective map\footnote{Note that obtaining this injectivity is the reason we insist that $w(\sigma) \prec w(\eta)$.} - that is, if $R(\sigma, \eta,m,k) = R(\sigma', \eta', m',k')$ for $k, k' \in [2:L-2]$, then $\sigma = \sigma'$, $\eta = \eta'$, $m=m'$ and $k=k'$. For $k \in [1:L]$, let
\be \label{DefMinimalPaths}
\mathcal{R}_{k} = \{ \zeta \in \Omega_{[1:L]^{2}} \, : \, \exists  \, \sigma, \, \eta, \, m \, \, \text{ s.t. } \, \, \zeta = R(\sigma,\eta,m,k) \},
\ee
so that $\mathcal{R} \equiv \cup_{k=1}^{L} \mathcal{R}_{k}$ is the collection of points  along these minimal-length and -energy paths between ground states (and $\mathcal{R}_{0} = \mathcal{R}_{L} = \cG$). Note that, for distinct $k, k' \in [2:(L-2)]$, $\mathcal{R}_{k} \cap \mathcal{R}_{k'} = \emptyset$. Thus, for $\zeta \in \cup_{k=2}^{L-2}\mathcal{R}_{k}$, let $\sigma(\zeta), \eta(\zeta), m(\zeta), k(\zeta)$ be the unique elements that satisfy
\be \label{EqDefZetaBackFunc}
\zeta = R(\sigma(\zeta),\eta(\zeta),m(\zeta),k(\zeta)).
\ee
When $\zeta \in \mathcal{R}_{1} \cup \mathcal{R}_{L-1}$, there may be several choices that satisfy Eq. \eqref{EqDefZetaBackFunc}, since these configurations have a single spin flipped relative to an element of the ground state; this single spin might be the first or last spin to be flipped in a path that is flipping a row or column.  Denote by $\phi(\zeta)$ the number of these choices, and denote the choices themselves by $(\sigma_{1}(\zeta), \ldots, k_{1}(\zeta)), \ldots , ( \sigma_{\phi(\zeta)}(\zeta),\ldots, k_{\phi(\zeta)}(\zeta))$, in any arbitrary fixed order.

Finally, let $\mathcal{R}_{\mathrm{init}} = \mathcal{R}_{1} \cup \mathcal{R}_{L-1}$ be the set of configurations which are adjacent to a ground state (a single spin-flip from a ground state), and $\mathcal{R}' = \mathcal{R} \backslash (\cG \cup \mathcal{R}_{\mathrm{init}})$.

\end{defn}

\begin{remark} \label{RemMinEnPath}
For fixed $m \in L$, the sequence $\{ R(\sigma,\eta,m,k) \}_{k=0}^{L}$ gives one of many minimal-length paths from $\sigma$ to $\eta$: that is, $R(\sigma,\eta,m,0) = \sigma$, $R(\sigma,\eta,m,L) = \eta$, and $R(\sigma,\eta,m,k)$ is obtained by flipping $k$ adjacent spins of the column on which $\sigma, \eta$ disagree. 
Both $R(\sigma,\eta,m,1)$ and $R(\sigma,\eta,m,L-1)$ belong to $\cR_{\mathrm{init}}$ and contain four defects at the vertices of a unit square.
Roughly speaking: $m$ identifies where in the column we start ``flipping" from $\sigma$ to $\eta$, and $k$ tells us ``how far" we are in the path from $\sigma$ to $\eta$. See Figure \ref{FigDominantTransitionsBetweenGroundStates} for a sample pair of ground states and an element along the path between them. 
\begin{figure}[h]
\includegraphics[scale = 0.7]{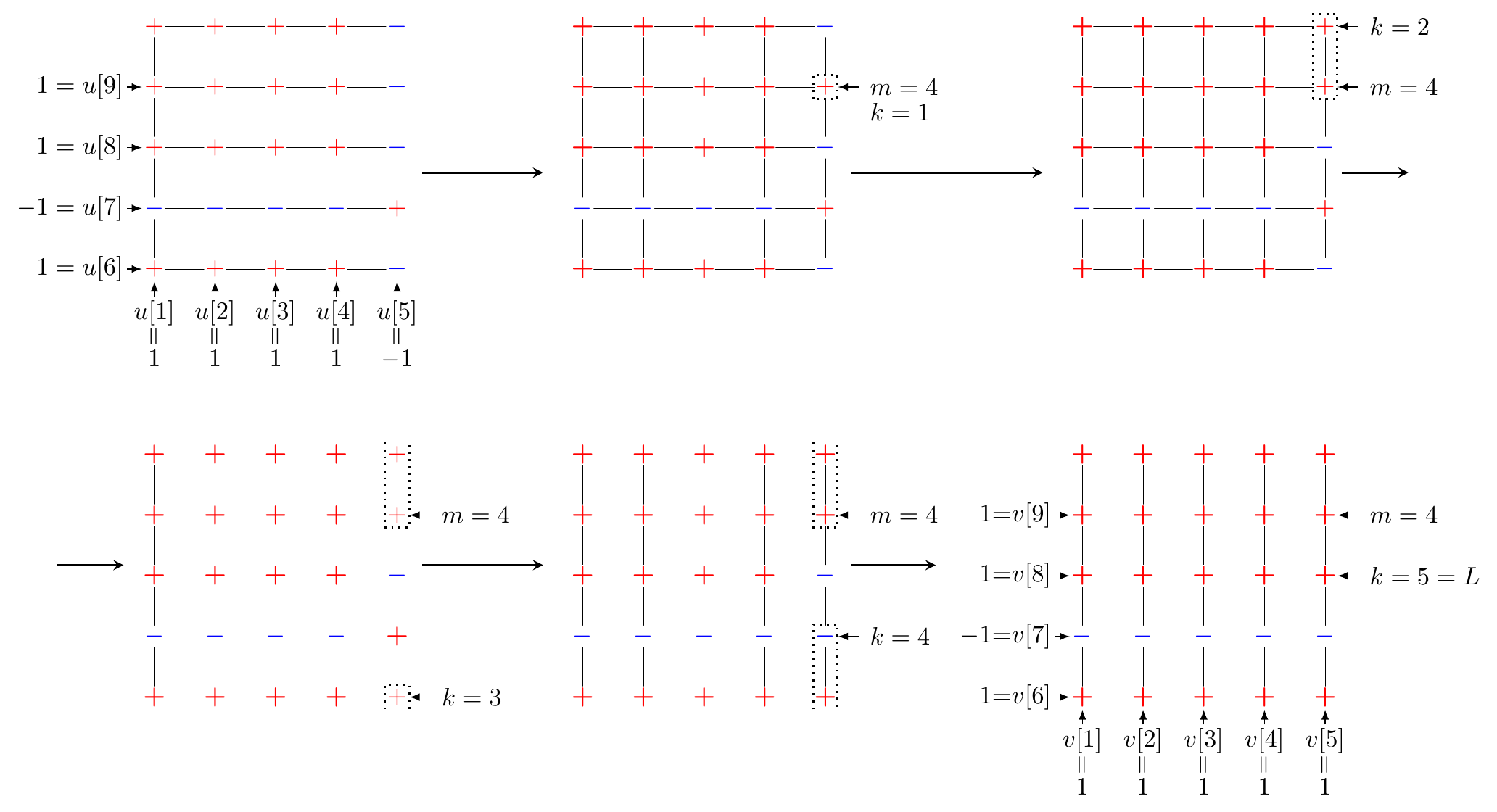}
\caption{\label{FigDominantTransitionsBetweenGroundStates} A minimal-length path $\{R(\s,\h,m,k)\}_{k=0}^{L}$ between two ground states $\s = w^{-1}(u)$ and $\h=w^{-1}(v)$, where $u$ and $v$ are given in the first and final frames.}
\end{figure}
\end{remark}

\subsection{Coupling to Simple Random Walks}
\label{RemSRWCouplingExplicit}

We will see that the process does not leave $\cR$ on the timescale of interest (see Corollary \ref{cor:noescape}). For this reason, it is sufficient to consider the dynamics restricted to $\cR$. 
It turns out that on $\cR'$ we can couple the process with a sequence of independent random walks, and we use this to build a description of the process on $\cR$ over the following sections.

\subsubsection{Coupling Construction} \label{SubSubSecInitCoupCon}

We begin with an informal description. Recall that $\cR$ is decomposed as the disjoint union $\cR = \cG \sqcup \cR_{\mathrm{init}} \sqcup \cR'$. The elements of $\cG$ are isolated from each other, and the only possible transitions from $\cG$ end in $\cR_{\mathrm{init}}$; furthermore, it is easy to compute the precise rate of these $\cG \rightarrow \cR_{\mathrm{init}}$ transitions (see below). Next, note that the elements of $\cR_{\mathrm{init}}$ are also isolated, and the only transitions of interest are to $\cG$ or $\cR'$ (transitions to $\cR^{c}$ are possible, but occur at a negligible rate over the relevant timescale). The elements of $\cR'$ are not isolated, but the dynamics on each ``connected component" can be easily described in terms of simple random walk on $[1:(L-1)]$. In particular, every  configuration $\zeta \in  \mathcal{R}'$ consists of two pairs of adjacent defects (initially distance $2$ when entering from $\cR_{\mathrm{init}}$). Over time, the signed distance between these particles \textit{exactly} undergoes a random walk on $[1,L-1]$ until the first time the signed distance is either $1$ or $L-1$ (the coupled walk first enters $1$ or $L-1$ exactly when the process first returns to $\cR_{\mathrm{init}}$). Each time the process goes from $\cR_{\mathrm{init}}$ to $\cR'$, we start a new coupling to a new random walk. At this point, we observe that we have given a \textit{complete} description of the dynamics of $\{X_{t}\}_{t \geq 0}$ until the first time that the process exits $\cR$ in terms of a few exponential random variables (the exit times from $\cG$ and $\cR_{\mathrm{init}}$) and random walks on the path $[1:(L-1)]$. The associated distributions are all very well-understood.

We introduce notation to formalize the above paragraph. Consider $\s \in \cG$; the total escape rate from $\s$ is $q_L = L^2 e^{-4 \b} \approx L^{-6}$, and all the possible transitions are to configurations in $\cR_{\textrm{init}}$. That is,
\begin{align*}
  -\cL(\s,\s) = \sum_{\h \in \cR_{\textrm{init}}}\cL(\s,\h) = q_L \,,
\end{align*}
and for each $\h \in \cR_{\textrm{init}}$ we have $\cL(\s,\h) \in \{0, e^{-4 \b}\}$.

Now consider $\s \in \cR_{\textrm{init}}$. This configuration moves to the single adjacent element of $\cG$ at rate $1$, to each of the four adjacent elements of $\cR'$ at rate $1$, and to the $L^{2} - 5$ adjacent elements of $\cR^{c}$ at total rate $4 e^{-2 \beta} + (L^{2}-9) e^{-4 \beta} \approx e^{-2 \beta}$.

Finally, we describe the excursions from $\cR_{\mathrm{init}}$ to $\cR_{\mathrm{init}} \cup \cR^{c}$ - that is, the law of the process from its first exit of a point in $\cR_{\mathrm{init}}$ to its next entrance to a point in $\cR_{\mathrm{init}} \cup \cR^{c}$.
 Fix $m \in [1:L]$ and a point of the form $X_{0} = \mathcal{R}(\sigma,\eta,m,2) \in \cR'$ or $X_0 = \mathcal{R}(\sigma,\eta,m,L-2) \in \cR'$. Recall that any element of $\cR'$ that is adjacent to an element of $\cR_{\rm init}$ has a unique representation of this form. Let $\{X_{t}\}_{t \geq 0}$ be the usual defects process, started at $X_{0}$. Next, let $\{L_{t}, U_{t}\}_{t \geq 0}$ be two independent simple random walks on $\mathbb{Z}$ with rate 1 and starting points 
 \begin{align*}
 L_0 =m\,, \quad \textrm{and} \quad U_0 =\begin{cases}
 m+2 & \textrm{if } X_0 = \mathcal{R}(\sigma,\eta,m,2)\,,\\
 m-2 & \textrm{if } X_0 = \mathcal{R}(\sigma,\eta,m,L-2)\,.
 \end{cases}
 \end{align*}
 Define the ``signed difference" random walk
\be
D_{s} \equiv U_{s} - L_{s};
\ee
note that this is a simple random walk on $\mathbb{Z}$ with rate 2. Define the stopping times
 \be
\tau_{\rm esc} &= \inf \{ s > 0 \, : \, X_{s} \notin \mathcal{R}' \}\\
\tau_{\rm flip} &= \inf \{ s > 0 \, : \, |D_{s}| = L-1\} \\
\tau_{\rm ret} &= \inf \{s > 0 \, : \,  |D_{s}| = 1 \}\\
\tau_{\min} &= \min \{ \tau_{\rm flip}, \tau_{\rm ret} \}.
\ee
Inspecting the generators of these processes, it is possible to couple $\{(X_{t},L_{t},U_{t})\}_{t \geq 0}$ so that
\be
X_{t} = \mathcal{R}(\sigma,\eta, L_{t}, D_{t} \bmod L)
\ee
for all $0 \leq t \leq \tau_{\rm esc} \leq \tau_{\min}$ (it is clear that $\tau_{\rm esc} \leq \tau_{\min}$ under this coupling).

With a slight abuse of notation we right $\bbP$ and $\bbE$ for the law and expectation of the coupling described above. We denote by $\mu_{\mathrm{srw}}$ the distribution of $\tau_{\rm esc}$, and $p_{\rm flip} = \P[\tau_{\rm flip} < \tau_{\rm ret} ]$. Recalling some standard facts about random walk (see \textit{e.g.} Chapter 4 of \cite{durrett2019probability}), we have for some constant $0 < C < \infty$
\be \label{IneqEscDom}
\E[\tau_{\min}] &\lesssim L, \\
\P[\tau_{\min} > C k L^{2}] &\leq e^{-k}, \qquad \forall k \in \mathbb{N} \\
p_{\rm flip} &\approx L^{-1}\,.
\ee

Having described the \textit{lengths} of excursions from $\cR_{\mathrm{init}}$ to $\cR_{\mathrm{init}} \cup \cR^{c}$, we now describe their \textit{endpoints}. For $i \in [1:L]$, let $r_{i} = [1:L] \times \{i\}$ be the $i$'th row and $c_{i} = \{i \} \times [1:L]$ be the $i$'th column. For $\sigma \in \cG$, define the collection of neighbours of $\sigma$ to be the $2L$ elements of $\cG$ that differ from $\sigma$ by flipping a single row or column - that is:
\be
\mathcal{N}(\sigma) = \mathcal{N}_{r}(\sigma) \cup \mathcal{N}_{c}(\sigma) \equiv \{ \sigma^{r_{i}} \}_{i \in [1:L]} \cup \{ \sigma^{c_{i}} \}_{i \in [1:L]}.
\ee
Note that each triple $\sigma \in \cG$ and $i,j \in [1:L]$ defines a unique element $\sigma^{(i,j)} \in \cR_{\mathrm{init}}$. Furthermore, every element in $\cR_{\mathrm{init}}$ has \textit{exactly} one representation of this form.


Having completed initial calculations, we can give an easy bound on the probability of exiting $\mathcal{R}$ over the relevant timescale:
\begin{lemma}
\label{LemmaExcEst}
Let $X_{0} = \s \in \mathcal{R}'$. For this chain, let
\be
\tau_{\mathrm{good}} = \min \{t > 0 \, : \, X_{t} \in \mathcal{R}_{\mathrm{init}} \} \quad \textrm{and} \quad \tau_{\mathrm{bad}} =  \min \{t > 0 \, : \, X_{t} \in \cR^c \}
\ee
be the first time that the chain either returns to a neighbour of $\cG$ or escapes $\cR'$. Then
\be
\P[\tau_{\mathrm{bad}} < \tau_{\mathrm{good}}] \lesssim L^{-3}.
\ee
\end{lemma}

\begin{proof}
Recall that, for $0 \leq t \leq \min\{\tau_{\mathrm{good}}, \tau_{\mathrm{bad}}\}$, the chain $\{X_{t}\}_{t \geq 0}$ can be coupled to a 1-dimensional random walk $D_t$ as above.
Also under this coupling we observe that $\{\tau_{\mathrm{bad}} < \tau_{\mathrm{good}}\} = \{\tau_{\rm esc}<\tau_{\rm min}\}$.
Recall $\tau_{\mathrm{min}}$ is the hitting time of $\{1\}\cup\{L-1\}$ for $|D_t|$, and the rate of going from any configuration in $\mathcal{R}$ to a configuration in $\cR^c$ (i.e. a configuration with 6 or more defects) is $\approx e^{-2\b}=L^{-4}$. Thus, it follows that
\be
\P[\tau_{\mathrm{bad}} < \tau_{\mathrm{good}}] &= \P[\tau_{\mathrm{esc}} < \tau_{\mathrm{min}}] \leq \E\left[\int_{0}^{\tau_{\mathrm{min}}} \textbf{1}_{\tau_{\mathrm{esc}} = t}dt\right] \leq \E[\tau_{\mathrm{min}}] L^{-4} \lesssim  L^{-3}\,.
\ee
\end{proof}

We now consider the possibility of leaving $\cR$ on the ``relevant" time scale of $\beta^{2} e^{4 \beta}$. Lemma \ref{LemmaExcEst} says that, over the first $\lesssim L^{2.5}$ excursions from $\mathcal{R}_{\mathrm{init}}$, we in fact stay within $\mathcal{R}$ with probability at least $\gtrsim 1 - L^{-0.5}$. Furthermore, since excursions from any ground state start at rate $L^2e^{-4\beta}$, and each time they return to $\cR_{\mathrm{init}}$ they hit a ground state with probability $\gtrsim 1/5 - O(L^{-4})$, we observe that we need to consider only $\approx L^{2} \log(L)^2$ excursions over the relevant timescale with high probability. Taking a union bound over these excrusions, we conclude that started from a ground state we are unlikely to leave $\cR$ on the relevant times scale. This observation is summarized in the following corollary. 

\begin{corollary}
\label{cor:noescape}
 The process started from in $\cG$ does not leave $\cR$ on the timescale of interest, that is
 \begin{align}
   \sup_{\s \in \cG}\bbP_{\s}[\t_{R^c} < \b^2 e^{4 \b}] \lesssim e^{-\frac{\beta}{4}}\,.
 \end{align}
\end{corollary}

We now consider $\sigma \in \cG$, $i,j \in [1:L]$ and a process $\{X_{t}\}_{t \geq 0}$ started at $X_{0} = \sigma^{(i,j)} \in \cR_{\mathrm{init}}$. For this process, let 
\be
\tau_{\rm start} = \inf \{ t > 0 \, : \, X_{t} \neq X_{0} \}\,,\quad \textrm{and} \quad 
\tau_{\rm end} = \inf \{t \geq \tau_{\rm start} \, : \, X_{t} \notin \cR' \}.
\ee
By the above calculations (and symmetry),
\be  \label{EqManyBasicCoupFacts}
\P[X_{\tau_{\rm end}} = \sigma] &= \left(\frac{1}{5}-O(L^{-3})\right)\\
\P[X_{\tau_{\rm end}} \in \{\sigma^{x}\}_{x\in c_i\cup r_j}] &=\left( \frac{4}{5}-O(L^{-3})\right)(1 - p_{\rm flip}) \\
\P[X_{\tau_{\rm end}} \in \{\sigma^{c_{i} \backslash \{(i,k)\}} \}_{k \in [1:L]}] &= \left( \frac{2}{5}-O(L^{-3})\right)p_{\rm flip} \\
\P[X_{\tau_{\rm end}} \in \{\sigma^{r_{j} \backslash \{(k,j)\}} \}_{k \in [1:L]}] &= \left( \frac{2}{5}-O(L^{-3})\right)p_{\rm flip}. \\
\ee

\subsubsection{Summary of Dynamics on $\cR$}

Using the description of the process in Section \ref{SubSubSecInitCoupCon}, we have learned the following about $\{ X_{t}\}_{t \geq 0}$ and its trace  $\{ \hat{X}_{t}\}_{t \geq 0}$ on $\cG$ started at $X_{0} = \sigma \in \cG$ (also summarized in Figure \ref{Figtypical_excursions}):

\begin{figure}[h]
\includegraphics[width = 0.6\textwidth]{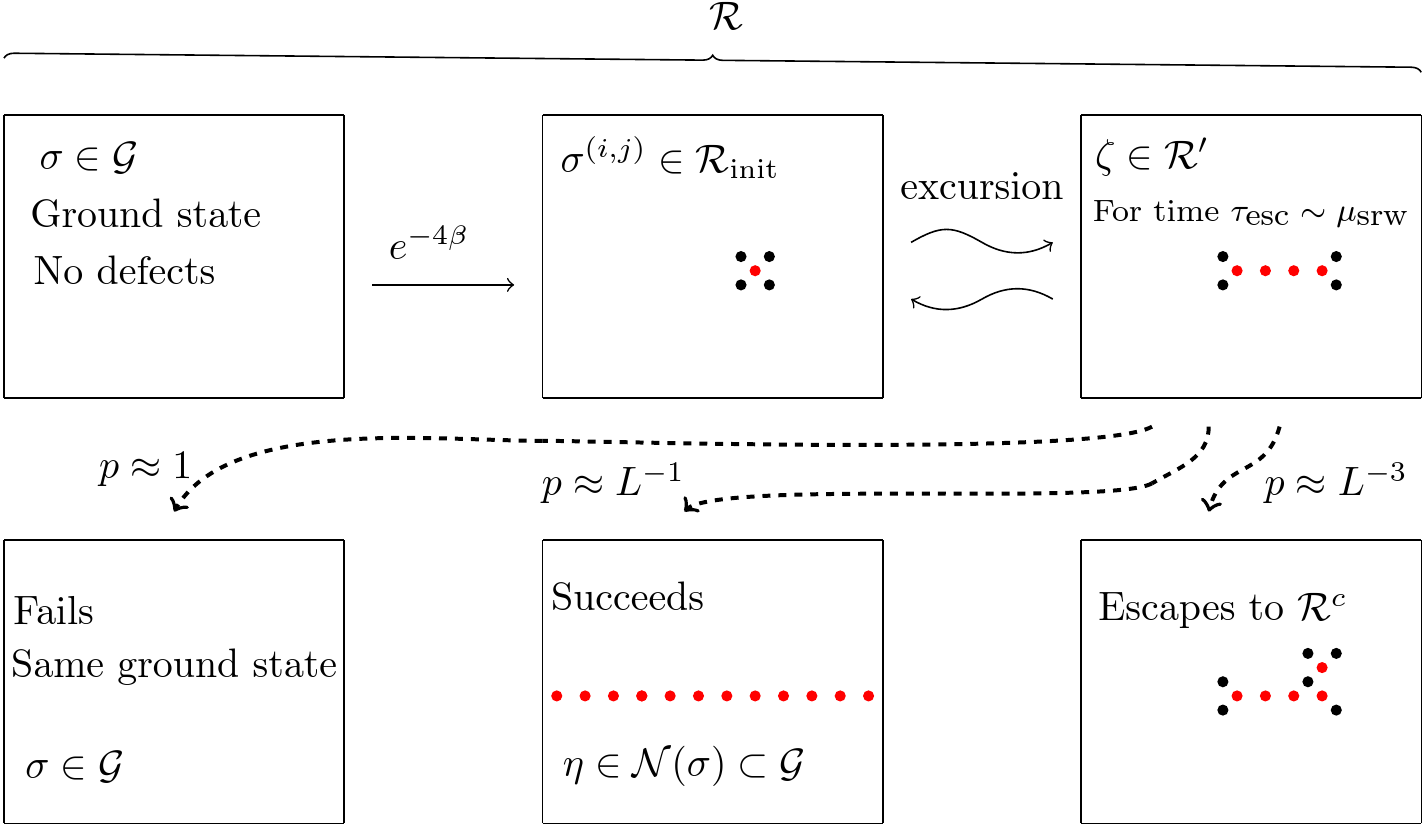}
\caption{\label{Figtypical_excursions} A cartoon of the types of transitions that are most likely during an excursion. }
\end{figure}

\begin{enumerate}
\item $X_{t}$ first exits $\{ \sigma \}$ to a point of the form $\sigma^{(i,j)} \in \cR_{\mathrm{init}}$.
\item $X_{t}$ then takes some number of excursions (possibly 0) from $\cR_{\mathrm{init}}$ to $\cR_{\mathrm{init}} \cup \cR^{c}$ before entering $\cG$. Each of these excursions takes an amount of time sampled from $\mu_{\rm srw}$, and results in returning to an element of $\{\sigma^{x}\}_{x \in c_i \cup r_j}$  with probability $\approx 1 - p_{\rm flip}$ and moving to an element of $\{\sigma^{c_{i} \backslash \{(i,k)\}}, \sigma^{r_{j} \backslash \{(k,j)\}} \}_{k \in [1:L]}$  with probability $\approx p_{\rm flip}$. We call one of these latter moves \textit{completing a flip}, as we end up at a point that has flipped all-but-one point in a single row or column, and is therefore close to a ground state which is in $\cN(\s)$.
\item Each time the excursion returns to  $\cR_{\mathrm{init}}$, the process then moves to the unique adjacent point in $\cG$ with probability $\frac{1}{5} - O(L^{-4})$, and so the \textit{total} rate at which the trace $\{\hat{X}_{t}\}_{t \geq 0}$ of $\{X_{t}\}_{t \geq 0}$ on $\cG$ moves is
\be \label{EqTotRate}
\Theta(L^{2}e^{-4 \beta} p_{\rm flip}) = \Theta(L^{-7}).
\ee
By symmetry, this trace process moves from $\sigma$ to each element of $\mathcal{N}(\sigma)$ at the same rate, which we denote by $a_{L}$. By \eqref{EqTotRate} and the fact that $|\mathcal{N}(\sigma)| = 2L$, we have shown that in fact $a_{L} \approx L^{-8}$. Since an excursion from $\sigma \in \cG$ to $\cG$ that \textit{doesn't} end in $\{\sigma\} \cup \mathcal{N}(\sigma)$ must ``complete at least two flips," and ``completing a flip" occurs with probability $p_{\rm flip} \approx L^{-1}$, we have also shown that the \textit{total} transition rate to  $\cG \backslash (\sigma \cup \mathcal{N}(\sigma))$ is $\Theta(L^{-8})$. Thus, describing the trace dynamics up to an approximation error of $L^{-1}$, we have
\be \label{EqFirstOrderAppr}
Q_{\cG}(\sigma,\sigma^{r_{i}}) = Q_{\cG}(\sigma,\sigma^{c_{i}}) = a_{L} &\approx L^{-8}, \qquad \textrm{for } \ i \in [1:L]\,, \\
Q_{\cG}(\sigma, \mathcal{N}(\sigma)^{c}) & \lesssim L^{-8}\,,
\ee
where we have used Corollary \ref{cor:noescape}.
\item We've noted that, from an element $\sigma^{(i,j)} \in \cR_{\mathrm{init}}$, the process will complete a flip and  move to an element in $\{\sigma^{r_{i} \backslash \{(k,i)\}} \}_{k \in [1:L]}$ \textit{before} leaving $\cR\setminus \cG$ with probability $\approx p_{\rm flip} \approx L^{-1}$.
 Thus, the probability of completing \textit{two} flips (that is,  moving to \textit{any} element of the form $\sigma^{r_{i} \cup c_{k} \backslash \{ x \}}$ or  $\sigma^{c_{j} \cup r_{k} \backslash \{x\}}$ for some $x \in [1:L]^{2}$) before leaving $\cR \backslash \cG$ is $\approx p_{\rm flip}^{2} \approx L^{-2}$. Define $\cN'(\sigma) =  \{ \sigma^{r_{i}\cup c_k} \}_{i,k \in [1:L]} \cup \{ \sigma^{c_{i}\cup r_k} \}_{i,k \in [1:L]}$, again by symmetry, the trace process moves from $\sigma$ to each element of $\cN'(\sigma)$ at the same rate, which we denote by $a_{L}'$. Noting that $|\cN'(\s)| \lesssim L^2$, we can thus refine the first-order approximation \eqref{EqFirstOrderAppr} to the second-order approximation
\be
Q_{\cG}(\sigma,\sigma^{r_{i}}) = Q_{\cG}(\sigma,\sigma^{c_{i}}) &= a_{L}, \qquad \qquad i \in [1:L]\,, \\
Q_{\cG}(\sigma, \eta) &=a_{L}'\lesssim L^{-9}, \qquad \eta \in \cN'(\s)\,, \\
Q_{\cG}\Big(\sigma, \big(\mathcal{N}(\sigma) \cup \cN'(\s)\big)^{c}\Big) &\lesssim L^{-9}\,,
\ee
where again we have used Corollary \ref{cor:noescape}.

\item Of course it is possible to complete three or more flips (or exit $\cR$ entirely) before returning to $\cG$, but the above bounds show that this occurs only with probability $O(L^{-3})$ and so is negligible on our time scale.
\end{enumerate}

This summary has two immediate corollaries:.

\begin{corollary} \label{LemmaHyperComp}
There exists a sequence of transition rate matrices $Q_{\rm hyp}$ on the sequence of state spaces $\{-1,+1\}^{2L-1}$ such that the off diagonal elements ($u\neq v$) satisfy:
\begin{align*}
Q_{\rm hyp}(u,v) =
\begin{cases}
 a_{L}\approx L^{-8}, & |u-v| = 1 \text{ or } u = -v\,,\\
a_{L}' \lesssim L^{-9}, &|u-v| = 2 \text{ or } |u-v| = 2L-2 \,, \\
 0, &\textrm{otherwise,}
\end{cases}
\end{align*}
and also
\be \label{IneqLemmaHyperCompMainConc}
\sup_{u \neq v \in \{-1,+1\}^{2L-1}} | Q_{\mathcal{G}}^{\rm per} (w^{-1}(u),w^{-1}(v)) - Q_{\rm hyp} (u,v) | \lesssim L^{-9},
\ee
where $w$ is as in Definition \ref{DefBijectionPeriodic}.
\end{corollary}

The above couplings describe all excursions from $\cG$ to $\cG$ that don't leave $\cR$ in terms of simple random walks and geometric random variables. Since the walk doesn't leave $\cR$ until time $\beta^{2} e^{4 \beta}$ with high probability by Corollary \ref{cor:noescape}, this immediately implies the following weak concentration bound of the occupation time $s_{\cG}$ of $\cG$:

\begin{corollary} \label{LemmaExpectedExcursionPerGr}
For any sequence $e^{4 \beta} \lesssim T(L) \lesssim \beta^{2} e^{4 \beta}$ and any $\epsilon > 0$,
\be
\lim_{L \rightarrow \infty} \sup_{\sigma \in \cG} \P[s_{\cG}(T(L)) < (1 - \epsilon) T(L)] = 0.
\ee
\end{corollary}

\subsection{Main Bounds}

We begin by checking that $Q_{\rm hyp}$ exhibits cutoff, just like the usual simple random walk on the hypercube, and that $Q_{\mathcal{G}}^{\rm per}$ inherits this cutoff. Denote by $\mu$ the uniform measure on $\{-1,+1\}^{2L-1}$ and, for $\delta > 0$, define the sequences of times
\be
T_{\delta,+} &=   \frac{1 + \delta}{4} a_{L}^{-1} \log(L) \,,  \\
T_{\delta,-} &=    \frac{1 - \delta}{4} a_{L}^{-1} \log(L) \,.
\ee

Denote by $P_{\rm hyp}^{t}$ and $P_{\cG}^t$ the laws of the Markov chain, at time $t$, associated with $Q_{\rm hyp}$ and $Q_{\cG}^{\rm per}$ respectively. The following implies that $Q_{\rm hyp}$ exhibits cutoff:

\begin{lemma} \label{LemmaCutoffSRWWithJump}

For all $\delta > 0$, we have
\be \label{IneqCutoffQUp}
\lim_{L \rightarrow \infty} \max_{u \in \{-1,+1\}^{2L-1}} \| P_{\rm hyp}^{T_{\delta,+}}(u,\cdot) - \mu(\cdot) \|_{\mathrm{TV}} = 0
\ee
and 
\be \label{IneqLowerBoundCutHyp}
\lim_{L \rightarrow \infty} \|  P_{\rm hyp}^{T_{\delta,-}}(\mathbf{+1},\cdot) - \mu(\cdot) \|_{\mathrm{TV}} = 1.
\ee
\end{lemma}

\begin{proof}
Denote by $Q$ the following generator of a random walk on $\{-1,+1\}^{2L-1}$:
\begin{align*}
Q(u,v) &= a_{L}, \qquad |u-v| = 1\\
Q(u,v) &= a_{L}', \qquad |u-v| = 2\\
Q(u,v) &= 0, \qquad \text{otherwise,} \\
\end{align*}
and denote by $P^{t}$ the law of the associated Markov chain at time $t$. Using the obvious coupling between Markov chains generated by $Q_{\rm hyp}$ and $Q$, we have for all $u \in \{-1,+1\}^{2L-1}$, $A \subset \{-1,+1\}^{2L-1}$ and $t \geq 0$
\be
| P_{\rm hyp}^{t}(u,A) - \mu(A) | \leq \max \{| P^{t}(u,A) - \mu(A) |, | P^{t}(-u,A) - \mu(A) |\}.
\ee
Applying Theorem 2.4.2 of \cite{saloff1997lectures}\footnote{Theorem 2.4.2 of \cite{saloff1997lectures} is written for the case $a_{L}' \equiv 0$. However, due to the fact that this is a random walk on the Cayley graph of an Abelian group, adding additional moves corresponding to $a_{L}' > 0$ cannot \textit{increase} the mixing time compared to the $a_{L}' \equiv 0$ case.} gives the upper bound on the mixing time in Equation \eqref{IneqCutoffQUp}.

To prove the Equation \eqref{IneqLowerBoundCutHyp}, consider the distinguishing statistic $f(x) = \big|\sum_i x_i\big|$ for $x \in \{-1,+1\}^{2L-1}$. Inequality \eqref{IneqLowerBoundCutHyp} can be checked easily by reading the proof of the analogous result for the usual random walk on the hypercube, as in \textit{e.g.} the bound in Inequality (7.27) of \cite{LPW09}. In re-reading the proof of Inequality (7.27) of \cite{LPW09}, note that $f(-x) = f(x)$, so that the extra `all-flip' transition which is not present in the usual random walk on the hypercube does not affect the calculation.\footnote{Note that \cite{LPW09} works in discrete time, our result follows immediately from the discrete-time result and standard bounds on the concentration of Poisson random variables.}
This completes the proof.

\end{proof}

From this, we deduce the following cutoff phenomenon for the trace process. Let $\tilde \mu = \mu\circ w$ denote the uniform measure on the set of ground states.

\begin{corollary} \label{LemmaTraceMixing}
We have
\be
\lim_{L \rightarrow \infty} \max_{\s \in \cG} \| P_{\cG}^{T_{\delta,+}}(\s,\cdot) - \tilde\mu (\cdot) \|_{\mathrm{TV}} = 0
\ee
and 
\be
\lim_{L \rightarrow \infty} \| P_{\cG}^{T_{\delta,-}}(+,\cdot) - \tilde\mu(\cdot) \|_{\mathrm{TV}} = 1.
\ee
\end{corollary}

\begin{proof}
By  Corollary \ref{LemmaHyperComp} and Lemma \ref{LemmaCutoffSRWWithJump},
\be
\lim_{L \rightarrow \infty} \max_{\s \in \cG} \| P_{\cG}^{T_{\delta,+}}(\s,\cdot) - \mu(\cdot) \|_{\mathrm{TV}} &\leq \lim_{L \rightarrow \infty} \max_{u \in \{-1,+1\}^{2L-1}} \| P_{\rm hyp}^{T_{\delta,+}}(u,\cdot) - \mu(\cdot) \|_{\mathrm{TV}} \\
& \qquad +  \lim_{L \rightarrow \infty} \max_{u \in \{-1,+1\}^{2L-1}} \| P_{\cG}^{T_{\delta,+}}(w^{-1}(u),\cdot) - P_{\rm hyp}^{T_{\delta,+}}(u,\cdot) \|_{\mathrm{TV}}\\
&\lesssim 0 +  \lim_{L \rightarrow \infty} T_{\delta,+} \max_{u \neq v \in \{-1,+1\}^{2L-1}} | Q_{\cG}(w^{-1}(u),w^{-1}(v)) - Q_{\rm hyp}(u,v) | \\
&\lesssim \lim_{L \rightarrow \infty} T_{\delta,+}  L^{-9} = 0.
\ee
By essentially the same calculation, observing that $w^{-1}(+\mathbf{1}) = + \in \cG$,
\be
\lim_{L \rightarrow \infty} \| P_{\cG}^{T_{\delta,-}}(+,\cdot) - \mu(\cdot) \|_{\mathrm{TV}} = 1.
\ee
This completes the proof.
\end{proof}

Fix $\sigma \in \Omega_{[1:L]^{2}}$ and let $\{X_{t}\}_{t \geq 0}$ be a copy of the original Markov process started at point $X_{0} = \sigma$. For $A \subset \Omega_{[1:L]^{2}}$, denote by
\be
\tau_{A} = \inf \{ t > 0 \, : \, X_{t} \in A \}
\ee
the hitting time of the set $A$.  Using the upper bound on the mixing time in Theorem \ref{ThmMainResPlus}, we have:

\begin{lemma} \label{LemmaHittingGround}
The hitting time $\tau_{\cG}$ satisfies
\be
\E_{\sigma}[\tau_{\cG}] \lesssim \beta^{9} e^{3.75 \beta},
\ee
uniformly in $\sigma \in \config$.
\end{lemma}
\begin{proof}

Denote by $\T_{\mathrm{mix}, d}^{\mathrm{per}}$ and $T_{\mathrm{mix}, d}^{+}$ the mixing times of the defect dynamics given by the generators $\cQ_{[0:L]^{2}}^{(\mathrm{per})}$ and $\cQ_{[1:L]^{2}}^{+}$ respectively. Since the defect dynamics are a deterministic function of the spin dynamics, it is clear that
\be
T_{\mathrm{mix}, d}^{+} \leq T_{\mathrm{mix}}^{+}, \quad T_{\mathrm{mix}, d}^{\mathrm{per}} \leq T_{\mathrm{mix}}^{\mathrm{per}}.
\ee
By the tightness of the spectral profile (as stated in Theorem 1 of \cite{kozma2007precision}) and since the transition rates for the defect dynamics with all plus boundary conditions are dominated by those with periodic boundary conditions  (see the proof of Lemma 7.8 in \cite{ChlebounSpm2018}), and the stationary distribution of defects is the same under both dynamics, we also have
\be
T_{\mathrm{mix}, d}^{\mathrm{per}} \lesssim \beta T_{\mathrm{mix}, d}^{+}.
\ee
Combining these two bounds with Theorem \ref{ThmMainResPlus},
\be \label{IneqPerDefMixing}
T_{\mathrm{mix}, d}^{\mathrm{per}} \lesssim \beta T_{\mathrm{mix}, d}^{+} \lesssim \beta^{9} e^{3.75 \beta}.
\ee
Recalling that $\pi(\cG) = 1 - o(1)$ (by Lemma \ref{lem:dom}), Inequality \eqref{IneqPerDefMixing} and  Theorem 1 of \cite{Peres2015} immediately imply the result.

\end{proof}

We also have a related bound on the occupation time of $\cG$:
\begin{lemma} \label{LemmaDefinitelyInGS}
 For all $t \in [e^{3.5\b},\b^2e^{4\b}]$,
\be
\sup_{\s \in \cG}\P_{\sigma}[\{X_{t} \notin \cG\}] \lesssim L^{-0.5}.
\ee
\end{lemma}

\begin{proof}
Fix $\s \in \cG$, we first observe from Corollary \ref{cor:noescape} that
\begin{align}
\label{eq:inR}
\P_{\sigma}[\{X_{t} \notin \cG\}] \leq \P_{\sigma}[\{X_{t} \notin \cG\} \cap \{\tau_{\mathcal{R}^{c}} > t\}] + L^{-0.5}\,.
\end{align}
Let $\{Z_{i}^{(j)}\}_{i,j \in \mathbb{N}} \stackrel{i.i.d.}{\sim} \mu_{\rm srw}$ and $\{B_{i} \}_{i \in \mathbb{N}} \stackrel{i.i.d.}{\sim} Geom(\frac{1}{6})$. Recall from Section \ref{RemSRWCouplingExplicit} that $\mu_{\rm srw}$ is the distribution of the length of an excursion from $\cR_{\mathrm{init}}$ to $\cR_{\mathrm{init}}\cup \cR^c$, while the probability of going from $\cR_{\mathrm{init}}$ directly to $\cG$ is just below $\frac{1}{5}$, so that the distribution of $B_{i}$ stochastically dominates the total number of times that an excursion from $\cG$ to $\cG$ will ever enter $\cR_{\mathrm{init}}$ (see Equation \eqref{EqManyBasicCoupFacts}). Let $t_{1},t_{2},\ldots$ be the points of a Poisson process with rate $q_{L}\approx L^{-6}$, the total escape rate from an element of $\cG$. By this stochastic domination argument, we have the immediate bound

\be \label{IneqCouplingPoissApprox}
\P_{\sigma}[\{X_{t} \notin \cG\} \cap \{\tau_{\mathcal{R}^{c}} > t\}] \leq \P[t \in \cup_{k} [t_{k}, t_{k} + \sum_{j=1}^{B_{k}} Z_{k}^{(j)}]].
\ee
Standard concentration inequalities for Poisson random variables give estimates much stronger than the following upper bound
\be  \label{ConcPoissIneq}
\P[|\{k \, : \, t_{k} \leq t\}| > K] \lesssim L^{-2},
\ee
where $K = 10 q_{L} t$.
By \eqref{IneqEscDom} and standard concentration bounds on geometric random variables,
\be \label{ConcGeomIneq}
\P[\max_{1 \leq k \leq K} \sum_{j=1}^{B_{k}} Z_{k}^{(j)} \geq L^{2} \log(L)^{4}] \lesssim L^{-2}.
\ee
Combining Inequalities \eqref{eq:inR}, \eqref{IneqCouplingPoissApprox}, \eqref{ConcPoissIneq} and \eqref{ConcGeomIneq} completes the proof.

\end{proof}

We now put these results together:
\begin{proof} [Proof of Theorem \ref{ThmMainResPer}]

We consider $T_{0}, T$ satisfying $e^{4\beta} \leq T_{0} \leq T \leq \beta^2 e^{4\b}$. We have the bound
\be \label{IneqUpperBoundCutoffFinalForm}
\max_{A \subset \cG, \sigma \in \cG}|\P_{\sigma}[X_{T} \in A] - \pi(A)| &\leq \max_{A \subset \cG, \sigma \in \cG} | \P_{\sigma}[\hat{X}_{s_{\cG}(T)} \in A] - \pi(A)| + \max_{\sigma \in \cG} \P_{\sigma}[X_{T} \notin \cG] \\
&\leq \sup_{T_{0} \leq t \leq T} \max_{A \subset \cG, \sigma \in \cG} |\P_{\sigma}[\hat{X}_{t} \in A] - \pi(A)| + \max_{\sigma \in \cG}\P_{\sigma}[s_{\cG}(T) < T_{0}] + \max_{\sigma \in \cG}\P_{\sigma}[X_{t} \notin \cG].
\ee
Fix $\delta > 0$, and in the above expression consider the sequence of times
\be
T_{0} = T_{0,+} \equiv \left\lceil \frac{1 + \delta}{4} a_{L}^{-1} \log(L) \right\rceil, \qquad T=T_{+} \equiv \left\lceil \frac{1 + 2\delta}{4} a_{L}^{-1} \log(L) \right\rceil\,,
\ee
and note that $T_0\approx T_+ \approx \beta e^{4\beta}$.
As $L$ goes to infinity, the first term in \eqref{IneqUpperBoundCutoffFinalForm} goes to 0 by Corollary \ref{LemmaTraceMixing}, the second term goes to 0 by Corollary \ref{LemmaExpectedExcursionPerGr}, and the third term goes to 0 by Lemma \ref{LemmaDefinitelyInGS}, so we have:
\be \label{IneqMoreCutoffUpperEnding1}
\lim_{L \rightarrow \infty} \max_{A \subset \cG, \sigma \in \cG}|\P_{\sigma}[X_{T} \in A] - \pi(A)| = 0.
\ee
Applying Lemma \ref{lem:dom} (Domination of Ground States), this gives
\be \label{IneqMoreCutoffUpperEnding2}
\lim_{L \rightarrow \infty} \max_{A \subset \config, \sigma \in \cG}|\P_{\sigma}[X_{T} \in A] - \pi(A)| &\leq \lim_{L \rightarrow \infty} \Big(\max_{A \subset \cG, \sigma \in \config}|\P_{\sigma}[X_{T} \in A] - \pi(A)| + \pi(\cG^{c})\Big) = 0.
\ee

Define the burn-in time $T_{b} = \lfloor e^{3.8 \beta} \rfloor$, which is negligible compared to $T$. By the strong Markov property, and monotonicity of the convergence of Markov chains to their stationary measures in total variation (see for example \cite{LPW09}), we have
\be 
\max_{A \subset \config, \sigma \in \config}|\P_{\sigma}[X_{T + T_{b}} \in A] - \pi(A)| \leq \max_{A \subset \cG, \sigma \in \cG} \max_{T \leq t \leq T + T_{b}}|\P_{\sigma}[X_{t} \in A] - \pi(A)| + \max_{\sigma \in \config} \P_{\sigma}[\tau_{\cG} > T_{b}].
\ee  
Applying Equality \eqref{IneqMoreCutoffUpperEnding2} and Lemma \ref{LemmaHittingGround}, both terms on the right-hand side of this expression go to 0 as $L$ goes to infinity. Since $T_{b} = o(T)$, this completes the ``upper bound" in our proof of cutoff. 

In the other direction, by the same argument as leading to \eqref{IneqUpperBoundCutoffFinalForm}, we have for every set $A \subset \cG$
\be \label{IneqLowerBoundCutoffFinalForm}
|\P_{+}[X_{T} \in A] - \pi(A)| \geq  \min_{T_{0} \leq t \leq T}  |\P_{+}[\hat{X}_{t} \in A] - \pi(A)| - \P_{+}[s_{\cG}(T) < T_{0}] - \P_{+}[X_{t} \notin \cG].
\ee
Now consider times
\be
T_{0} = T_{0,-} \equiv \left\lceil \frac{1 -2 \delta}{4} a_{L}^{-1} \log(L) \right\rceil, \qquad T=T_{-} \equiv \left\lceil \frac{1 -\delta}{4} a_{L}^{-1} \log(L) \right\rceil.
\ee
 As $L$ goes to infinity, there exists an $A \subset \cG$ such that the first term in \eqref{IneqLowerBoundCutoffFinalForm} goes to 1 by Corollary \eqref{LemmaTraceMixing}; the other terms go to 0 by the same argument that precedes Inequality \eqref{IneqMoreCutoffUpperEnding1}. Thus, for this choice of $T$, we have
\be \label{IneqMoreCutoffLowerEnding1}
\lim_{L \rightarrow \infty} \max_{A \subset \cG, \sigma \in \cG}|\P_{\sigma}[X_{T} \in A] - \pi(A)| = 1.
\ee
The theorem now follows immediately from Equalities \eqref{IneqMoreCutoffUpperEnding2} and \eqref{IneqMoreCutoffLowerEnding1}.

\end{proof}

\section{Acknowledgments}

We thank Alessandra Faggionato, Fabio Martinelli and Cristina Toninelli for their help, hospitality and encouragement.


\bibliographystyle{abbrv}
\bibliography{PlaqBib}

\end{document}